\documentclass[12pt,a4paper,twoside]{article}

\usepackage{amsmath}
\usepackage{amssymb}
\usepackage{amsthm}
\usepackage[utf8]{inputenc}
\usepackage[ruled,vlined]{algorithm2e}
\usepackage{graphicx}
\usepackage[justification=centering]{caption}
\usepackage{subfig}
\usepackage[export]{adjustbox}
\usepackage{setspace}
\usepackage{layout}
\usepackage{lineno,soul}
\usepackage[dvipsnames]{xcolor}
\PassOptionsToPackage{hyphens}{url}
\usepackage[citecolor=NavyBlue,colorlinks=true,linkcolor=NavyBlue,urlcolor=NavyBlue,breaklinks]{hyperref}
\usepackage[hyphenbreaks]{breakurl}
\usepackage{cleveref}
\usepackage[title]{appendix}
\usepackage{bm}
\usepackage{booktabs}
\usepackage[numbib,notlof,notlot,nottoc]{tocbibind}

\newcommand{\norm}[1]{\left\lVert #1\right\rVert}

\newcommand\bb[1]{\mathbb{#1}} 
\newcommand\q[1]{\mathcal{#1}} 
\newcommand\conj[1]{\overline{#1}}

\newcommand\RE{\operatorname{Re}}
\newcommand\IM{\operatorname{Im}}

\newcommand\MOD{\operatorname{mod}}

\newcommand\diag{\operatorname{diag}}
\newcommand\sgn{\operatorname{sgn}}

\DeclareMathOperator*{\argmin}{arg\,min}

\newtheorem{theorem}{Theorem}

\newtheorem{lemma}[theorem]{Lemma}
\newtheorem{proposition}[theorem]{Proposition}
\newtheorem{corollary}[theorem]{Corollary}

\newtheorem{remark}[theorem]{Remark}
\hypersetup{%
    bookmarksnumbered, bookmarksopen=true, bookmarksopenlevel=1,%
}

\numberwithin{theorem}{section} 

\DeclareMathAlphabet{\myds}{U}{bbold}{m}{n}

\title{Stochastic Amplitude Flow for phase retrieval, its convergence and doppelg{\"a}ngers}
\author{Oleh Melnyk\thanks{Mathematical Imaging and Data Analysis, Helmholtz Center Munich, 85764 Neuherberg, Germany (\href{mailto:oleh.melnyk@helmholtz-muenchen.de}{oleh.melnyk@helmholtz-muenchen.de}).} \thanks{Department of Mathematics, Technical University of Munich, 85746 Garching bei München, Germany.}}

\begin{document}

\maketitle
\begin{abstract}

In this paper, we focus on Stochastic Amplitude Flow (SAF) for phase retrieval, a stochastic gradient descent for the amplitude-based squared loss. While the convergence to a critical point of (nonstochastic) Amplitude Flow is well-understood, SAF is a much less studied algorithm. We close this gap by deriving the convergence guarantees for SAF based on the contributions for Amplitude Flow and analysis for stochastic gradient descent. These results are then applied to two more algorithms, which can be seen as instances of SAF. The first is an extension of the Kaczmarz method for phase retrieval. The second is Ptychographic Iterative Engine, which is a popular algorithm for ptychography, a special case of phase retrieval with the short-time Fourier transform.


\end{abstract}

\section{Introduction}\label{sec: introduction}

In recent years, a measurement scheme known as ptychography \cite{Pfeiffer.2018} became particularly popular in coherent diffraction imaging as it allows obtaining high-resolution images at atomic scale \cite{Chen.2021}. The core idea is to record a collection of diffraction patterns, each corresponding to a single illumination of a small region of a specimen. Then, the goal of ptychography is to recover the illuminated object from these images. 

Ptychography is related to the math problem of phase retrieval, which arises in imaging and various areas such as astronomy \cite{Fienup.1993}, biology \cite{Piazza.2014, Giewekemeyer.2015}, crystallography \cite{Lazarev.2018, Chen.2021} and material sciences \cite{Hoydalsvik.2014, Esmaeili.2015}. Commonly, phase retrieval is a recovery of an unknown vector $x \in \bb C^d$ from the measurements of the form
\begin{equation}\label{eq: meas}
y_j = |(Ax)_j|^2 + n_j, \quad j = 1,\ldots, m,
\end{equation}
where $A \in \bb C^{m \times d}$ is a measurement matrix and $n \in \bb R^{m}$ denotes perturbations. In case of ptychography, $A$ corresponds to a (subsampled) short-time Fourier transform (STFT), which is why ptychographic reconstruction is also referred to as STFT phase retrieval. 

Note that for arbitrary $\alpha \in \bb C$ with $|\alpha| = 1$, we have that 
\[
|A \alpha x| = |\alpha A x| = |A x|,
\]
and, thus, it is not possible to distinguish $\alpha x$ and $x$ from the measurements $\eqref{eq: meas}$. This is known as the global phase ambiguity. Therefore, unique recovery of $x$ up to a global phase corresponds to unique identification of the set $\{ \alpha x : |\alpha| = 1 \}$. This is only possible if the mapping $\{\alpha z : |\alpha| = 1\} \mapsto |Az|$ is injective, i.e., there is no other vector than $x$ (up to a global phase), which provides the same measurements $y$. It was established in \cite{Conca.2015} that $m \ge 4d-4$ measurements suffice for a unique recovery if $A$ is a generic matrix. An analogous result was obtained in \cite{Kabanava.2016} for the case when the rows of $A$ are independent subgaussian random vectors and the number of measurements satisfies $m \ge c d$ for some constant $c>1$. If $A$ is the STFT matrix, uniqueness is achievable for generic objects and illumination functions \cite{Bendory.2021}. In general, the possibility of unique recovery depends on $A$. Sometimes, it can be only guaranteed for a subset of vectors in $\bb C^d$ \cite{Alaifari.2021}. Furthermore, there are well-known cases when unique recovery is not possible, e.g., if $A$ is a discrete Fourier matrix \cite{Beinert.2015}.


Over the years, various methods were developed to recover $x$ from data of the form \eqref{eq: meas}. To name a few, there are gradient methods \cite{Xu.2018, Wang.2018}, projection methods \cite{Gerchberg.1972, Thao.2021}, alternating directions method of multipliers \cite{Chang.2018}, majorize-minimization algorithms \cite{Li.2022, Fatima.2022}, convex relaxations \cite{Candes.2013, Ghods.2018}, direct methods \cite{Alexeev.2014, Iwen.2016, Forstner.2020} and many more. 

In this paper, we focus on the Amplitude Flow (AF) algorithm \cite{Xu.2018, Wang.2018}, which is a gradient descent scheme 
\[
z^{t+1} = z^t - \mu_t \nabla_z \q L(z^t), \quad t \ge 0,
\]
with step sizes $\{ \mu_t \}_{t \ge 0}$ applied to the amplitude-based squared loss function
\begin{equation}\label{eq: loss}
\q L(z) := \sum_{j = 1}^m | |(Az)_j| - \sqrt{y_j} |^2, \quad z \in \bb C^d.
\end{equation}
In the noiseless case, the true object $x$ is a global minimizer of the loss function, which is the motivation for this algorithm. However, as $\q L$ is a nonsmooth nonholomorphic function of complex variables, it requires the use of generalized Wirtinger gradient for optimization. The loss function is nonconvex and, therefore, the algorithm can stop at any critical point. It was shown that AF reaches a critical point at a sublinear rate of convergence \cite{Xu.2018}. In case of a random matrix $A$, e.g., the entries of $A$ are independently identically distributed standard complex Gaussian vectors, it is possible to construct a good initialization $z^0$ via the spectral method \cite{Candes.2015}. With such initialization, the algorithm will converge to $x$ at a linear rate \cite{Wang.2018}.       

With more efficient detectors, the resolution of recorded diffraction patterns has grown, which significantly increased data volumes to be processed by algorithms. Thus, computational efficiency became an important feature when choosing an algorithm. A possible choice of a fast algorithm is Stochastic Amplitude Flow (SAF), a version of AF where gradients are replaced with stochastic gradients \cite{Wang.2017, Xiao.2020}. While the gradient is a sum of $m$ gradients, each corresponding to a summand in \eqref{eq: loss}, the stochastic gradient only samples a few of them at each iteration, which leads to a lower computing cost. An extreme case, where a single summand in \eqref{eq: loss} is sampled, is also known as the Kaczmarz method for phase retrieval \cite{Wei.2015, Wang.2017, Tan.2019, Romer.2021, Zhang.2021}. That is, an index $j \in [m]$ is sampled and the next iterate of the Kaczmarz method is constructed such that the corresponding error term $| |(Az)_j| - \sqrt{y_j} |^2$ in \eqref{eq: loss} is nullified. While both SAF and the Kaczmarz method converge to $x$ at a linear rate if $A$ is a random matrix, in general, convergence properties of these algorithms are not known. The gap in theoretical analysis between the random and nonrandom cases is especially prominent when compared to the studies in optimization on stochastic gradient descent for nonconvex functions \cite{Nemirovski.2009, Ghadimi.2013, Khaled.2020, Orabona.2020, Sebbouh.2021, Patel.2022, Liu.2022}.

With the increasing popularity of ptychographic recovery, some of these algorithms were revisited to address STFT phase retrieval \cite{Marchesini.2016, Bendory.2018, Melnyk.2022}. For instance, in \cite{Bendory.2018} it was shown that in case of ptychography, AF has guaranteed linear convergence to $x$ if the initial guess is in a neighborhood of $x$. At the same time, new algorithms were introduced specifically for STFT phase retrieval, for example \cite{Rodenburg.2004, Iwen.2016, Prusa.2017}. A particular popularity among practitioners was obtained by Ptychographic Iterative Engine (PIE) \cite{Rodenburg.2004} and its variations \cite{Maiden.2009, Maiden.2017}. The algorithm has several theoretical interpretations \cite{Thibault.2013, Maiden.2017}. One of them \cite{Maiden.2017} describes the iteration of PIE as a gradient step for a certain loss function corresponding to a single sampled diffraction pattern. Thus, the processing of only one diffraction pattern per iteration determines the computational efficiency of PIE. However, the parameter corresponding to the step size is commonly chosen empirically by trial and error, and, to our knowledge, there is no supporting convergence analysis of PIE in the literature. 


 
Turning to our contribution, it can be separated into two parts. Firstly, we derive a convergence theory for SAF for an arbitrary measurement matrix $A$. These results are based on recent studies in optimization \cite{Khaled.2020, Orabona.2020, Liu.2022}, which show convergence of stochastic gradient descent under minimal assumptions on the loss function and, in our case, can be adjusted to suit the nonsmooth loss function \eqref{eq: loss}.

Secondly, we discuss two algorithms, which can be seen as special cases of SAF. The Kaczmarz method can be considered as an instance of SAF with a constant step size. Furthermore, for STFT phase retrieval we show that PIE is an instance of SAF. Consequently, previously established results for SAF are used to find a proper choice of step sizes for PIE, such that convergence is guaranteed to at least a critical point. 

The paper is structured in the following way. In \Cref{sec: notation} we establish the notation used throughout the paper. Next, we review basic facts about phase retrieval and Amplitude Flow in \Cref{sec: preliminaries AF}. The main part of \Cref{sec: SAF} is about Stochastic Amplitude Flow and contains new results regarding its convergence, while subsections \Cref{sec: kaczmarz} and \Cref{sec: PIE} discuss special cases of SAF, the Kaczmarz method and PIE, respectively. All proofs are provided in \Cref{sec: proofs} followed by a short conclusion.    

\subsection{Notation}\label{sec: notation}
We will use the short notation $[a] = \{1,2,\ldots, a\}$ for index sets. The complex unit is denoted by $i$. The complex conjugate, the real part, the imaginary part and the absolute value of $\alpha \in \bb C$ are given by $\bar \alpha$, $\RE \alpha$, $\IM \alpha$ and $|\alpha|$, respectively. The transpose and complex conjugate transpose of a vector $v$ or a matrix $B$ are denoted by $v^T, v^*$ and $B^T, B^*$, respectively. For $1 \le p \le \infty$ the $\ell_p$-norm of a vector $v \in \bb C^a$ is given by 
\[
\norm{v}_p^p := \sum_{j \in [a]} |v_j|^p, \quad  1 \le p < \infty, 
\quad \text{ and } \quad 
\norm{v}_\infty = \max_{j \in [a]} |v_j|.
\]
The inner product of two vectors $u,v \in \bb C^a$ is denoted by $\langle u, v\rangle = \sum_{j \in [a]} u_j \conj v_j$. The spectral norm of a matrix $B \in \bb C^{b \times a}$ is 
\[
\norm{B} := \max_{v \in \bb C^{a},\ \norm{v}=1} \norm{B v}_2,
\text{ so that } \norm{B v}_2 \le \norm{B} \norm{v}_2 \text{ for all } v \in \bb C^a.  
\] 
The Frobenius norm of $B$ is given by $\norm{B}_F = \sqrt{ \sum_{j \in [a]} \sum_{k \in [b]} |B_{k,j}|^2 }$.
The identity matrix is denoted by $I_a$. A diagonal matrix $\diag(v) \in\bb C^{a \times a}$ is formed by placing the entries of the vector $v \in \bb C^{a}$ onto main diagonal, that is 
\[
\diag(v)_{k,j} := 
\begin{cases}
v_k, & k = j, \\
0, &  k \neq j.
\end{cases}
\] 
The spectral norm of $\diag(v)$ is given by the $\ell_\infty$-norm of the vector $v$. The discrete Fourier transform matrix $F \in \bb C^{d \times d}$ is given by 
\begin{equation}\label{eq: DFT}
F_{k,j} =e^{-2 \pi i (k-1)(j-1)/d}, \quad k,j \in [d].
\end{equation}
The matrix $\frac{1}{\sqrt d} F$ is unitary, that is $\norm{\frac{1}{\sqrt d} F v}_2 = \norm{v}_2$ for all $v \in \bb C^d$ and the inverse of $F$ admits
\begin{equation}\label{eq: fourier prop}
F^{-1} = \tfrac{1}{d} F^*.
\end{equation}
The circular shift operator $S_r \in \bb C^{d \times d}$, $r \in [d]$ acts on a vector $v \in \bb C^d$ as
\begin{equation}\label{eq: shift op}
(S_r v)_j := v_{(j - 1 - r \MOD d) + 1}, \quad j \in [d].
\end{equation}

For a function $f: \bb C \to \bb C$ and a vector $v \in \bb C^a$ by $f(v)$ we will denote the entrywise application of $f$. The common use will be the absolute value $|v|$, square $v^2$ and square root $\sqrt v$, an addition of a constant $v + \alpha$ for some $\alpha \in \bb C$ and a sign function $\sgn(v)$ with
\[
\sgn(\alpha)  = 
\begin{cases}
\alpha/|\alpha|,& \alpha \neq 0, \\
0, & \alpha = 0.
\end{cases}
\]
Furthermore, $u \circ v$ and $u / v$ denote entrywise (Hadamard) product and division of two vectors $u, v \in \bb C^{a}$.

For a random variable $X$ its expectation is denoted by $\bb E X$. We say that a random event $B$ happens almost surely (a.s.) if the probability of this event is one.

Consider two sequences $\{ \alpha_t \}_{t \ge 0}, \{ \beta_t \}_{t \ge 0}$. We say that $\alpha_t$ is $o(\beta_t)$ as $t \to \infty$ if $\displaystyle\lim_{t \to \infty} \tfrac{\alpha_t}{\beta_t} = 0$. The sequence $\alpha_t$ converges to $\alpha \in \bb C$ at a linear rate if $\lim_{t \to \infty}  \frac{\alpha_{t+1} - \alpha}{\alpha_{t} - \alpha} = c$ for some $0< c < 1$. If $c= 1$,  then the convergence rate is called sublinear.  

\section{Amplitude Flow}\label{sec: preliminaries AF}

With the notation introduced above, \eqref{eq: meas} can be written as
\[
y = |Ax|^2 + n.
\]
As it was mentioned in \Cref{sec: introduction}, the Amplitude Flow (\ref{eq: AF}) algorithm is a first-order minimization of the loss function \eqref{eq: loss}. Note that $\q L$ is not differentiable everywhere. Thus, in order to apply a gradient method, we have to  consider first a smoothed objective function 
\begin{equation}\label{eq: loss smooth}
\q L_\varepsilon (z) = \norm{ \sqrt{|Az|^2 + \varepsilon} - \sqrt{y + \varepsilon} }, 
\end{equation}
with parameter $\varepsilon \ge 0$. For $\varepsilon >0$, the loss function $\q L_\varepsilon$ is twice continuously differentiable and, its Wirtinger gradient is given by
\[
\nabla_z \q L_\varepsilon (z) = A^*(A z - \sqrt{y + \varepsilon} \circ A z / \sqrt{ |Az|^2 + \varepsilon}  ),
\]
see \cite{Xu.2018, Melnyk.2022}. We provide an overview of Wirtinger derivatives in \Cref{sec: Wirtinger derivatives}. From \eqref{eq: loss smooth}, it follows that $\q L= \q L_0$ and its generalized Wirtinger gradient can be obtained by letting $\varepsilon$ tend to zero,
\begin{equation}\label{eq: generalized gradient}
\nabla_z \q L_0(z) := \lim_{\varepsilon \to 0} \nabla_z \q L_\varepsilon (z) = A^*( Az - \sqrt y \circ \sgn(Az) ).
\end{equation}

Based on these definitions, the gradient descent iteration of \ref{eq: AF} reads as
\begin{equation}\label{eq: AF} \tag{AF}
z^{t+1} = z^t - \mu_t \nabla_z \q L_\varepsilon(z^t), \quad t \ge 0,
\end{equation}
with initial guess $z_0 \in \bb C^d$ and step sizes $\{ \mu_t \}_{t \ge 0}$. The iterations of \ref{eq: AF} are repeated until a critical (fixed) point $z^{t+1} = z^{t}$ is reached. In the literature, the case $\varepsilon = 0$ is referred to as Amplitude Flow \cite{Xu.2018, Wang.2018}, while $\varepsilon>0$ is named as smoothed or perturbed Amplitude Flow \cite{Gao.2020, Luo.2020, Xiao.2021}.
 
If $A$ allows for a unique reconstruction and no noise is present, the set $\{\alpha x: |\alpha| = 1\}$ is the set of all global minimizers of $\q L_\varepsilon$, which motivates its minimization. However, since $\q L_\varepsilon$ is nonconvex, only convergence to a critical point can be expected, which is not necessarily a global minimum. In order to guarantee convergence and quantify its rate, an analogue of descent lemma \cite[Lemma 5.7]{Beck.2017} can be established. 

\begin{lemma}[Descent lemma]\label{l: descent lemma}
Let $z, v \in \bb C^d$ and $\varepsilon \ge 0$. Then, we have
\[
\q L_\varepsilon(z + v) \le \q L_\varepsilon(z) + 2 \RE( \langle \nabla_z \q L_\varepsilon(z), v \rangle) + \norm{A}^2 \norm{v}_2^2.
\]
Furthermore, if $\varepsilon >0$, the gradient $\nabla_z \q L_\varepsilon$ is Lipschitz-continuous with Lipschitz constant 
\begin{equation}\label{eq: Lipschitz constant}
L := \norm{A}^2 \max \{ 1, \norm{ y + \varepsilon }_\infty^{1/2} \varepsilon^{-1/2} -1 \},
\end{equation}
so that
\[
\norm{ \nabla_z \q L_\varepsilon(z) - \nabla_z \q L_\varepsilon(v) }_2 
\le L \norm{z - v}_2.  
\] 
\end{lemma}
This result can be deduced from a proof in \cite[pp. 26-29]{Xu.2018} and is a key component in the analysis of gradient methods for phase retrieval. 

Let 
\begin{equation}\label{eq: loss min}
\q L_\varepsilon^{\inf} := \inf_{z \in \bb C^d} \q L_\varepsilon(z) \ge 0
\end{equation}
be the infimum of $\q L_\varepsilon$. \Cref{l: descent lemma} can be used for iterations of \ref{eq: AF} to conclude that for a properly chosen constant step size, the algorithm will decrease the value of the loss function \eqref{eq: loss smooth}. Furthermore, the gradient will eventually vanish and the rate of convergence is sublinear.

\begin{theorem}[{\cite[Theorem 1]{Xu.2018}}]\label{thm: AF convergence}
Let $\varepsilon \ge 0$. Consider measurements $y$ of the form \eqref{eq: meas} with $\mu_t = \mu$ satisfying $0<\mu \le \norm{A}^{-2}$ and arbitrary $z^0 \in \bb C^d$. Then, for iterates $\{ z^t \}_{t \ge 0}$ defined by \ref{eq: AF} we have
\[
\q L_\varepsilon(z^{t+1}) \le \q L_\varepsilon(z^t) - \mu \norm{\nabla_z \q L_\varepsilon(z^t) }_2^2 \ \text{ for all } t \ge 0,
\]
\[
\norm{\nabla_z \q L_\varepsilon(z^t)}_2 \to 0,\ t \to \infty.
\]
Furthermore, if the number of iterations $T$ fulfills
\[
T \ge \gamma^{-2} \mu^{-1} (\q L_\varepsilon(z^0) - \q L_\varepsilon^{\inf})
\]
for some $\gamma>0$, then the norms of the gradients satisfy
\[ 
\min_{t= 0,\ldots, T-1} \norm{\nabla_z \q L_\varepsilon(z^t)}_2 \le \gamma.
\] 
\end{theorem}  

\section{Stochastic Amplitude Flow}\label{sec: SAF}

If the number of measurements is extremely large, the computation of the gradient $\nabla_z \q L_\varepsilon(z)$ becomes costly. One of possible solutions to this problem is to consider stochastic gradients instead. Consider a partition of $A$ into $R \in [m]$ row blocks, each denoted by $A_r,$ $r \in [R]$. Furthermore, let $y^r,$ $r \in [R]$, be the measurements corresponding to $A_r$, so that
\[
A = 
\begin{bmatrix}
A_1\\
\vdots\\
A_R 
\end{bmatrix},
\quad 
y = \begin{bmatrix}
y^1\\
\vdots\\
y^R 
\end{bmatrix}.
\]     
Then, $\q L_\varepsilon$ naturally splits into a sum of loss functions for each of the blocks,
\[
\q L_\varepsilon(z) = \norm{ \sqrt{|Az|^2 + \varepsilon} - \sqrt{y + \varepsilon} }_2^2 =  \sum_{r \in [R]} \norm{ \sqrt{ |A_r z|^2 + \varepsilon} - \sqrt{y^r + \varepsilon} }_2^2 =: \sum_{r \in [R]} \q L_{\varepsilon,r}(z).
\] 
Based on this decomposition, the stochastic gradient of $\q L_\varepsilon$ at a point $z \in \bb C^d$ is defined as
\begin{equation}\label{eq: stochastic gradient}
g_\varepsilon(z) = \frac{1}{K} \sum_{k \in [K]} \frac{1}{p_{r_k}} \nabla_z \q L_{\varepsilon,r_k}(z),
\end{equation}
where $K \in \bb N$ and the indices $r_1, \ldots, r_K$ are drawn independently at random with replacement from a distribution $p$. We require that $0< p_r < 1$, $r \in [R]$ and $\sum_{r \in[R]}p_r = 1$. In this paper, we will only consider sampling with replacement, while some other sampling schemes can be found in \cite[Section 4.5]{Khaled.2020}. The scaling factors $1/p_{r_k}$ are chosen in a way that \eqref{eq: stochastic gradient} is an unbiased estimate of the gradient.
\begin{proposition}[{Version of \cite[Proposition 3]{Khaled.2020}}]\label{p: gradient properties}
Let $\varepsilon \ge 0$. The stochastic gradient \eqref{eq: stochastic gradient} satisfies 
\[
\bb E g_\varepsilon(z) = \nabla_z \q L_{\varepsilon}(z).
\]
Furthermore, it admits
\[
\bb E \norm{ g_\varepsilon(z) }_2^2 \le \alpha (\q L_\varepsilon(z) - \q L_\varepsilon^{\inf}) 
+ \beta \norm{\nabla_z \q L_\varepsilon(z)}_2^2
+ \delta,
\]
with
\begin{equation}\label{eq: abc params}
\alpha = \frac{1}{K} \max_{r \in [R]} \frac{ \norm{A_r}^2 }{p_r}, \quad
\beta =  1 - \frac{1}{K}, \quad
\delta = \alpha \left[ \q L_\varepsilon^{\inf} - \sum_{r \in R} \q L_{\varepsilon,r}^{\inf} \right] \ge 0,
\end{equation}
where $\q L_{\varepsilon,r}^{\inf}$ are defined analogously to \eqref{eq: loss min}.
\end{proposition}

\begin{remark}
The results in \cite{Khaled.2020} are derived for functions of real variables, while we work with functions of complex variables. This leads to small changes in the constants $\alpha, \beta, \delta$ compared to the original results in \cite{Khaled.2020}. For completeness, we provide proofs for the complex case in \Cref{sec: complex sgd}.
\end{remark}

Now, by replacing the gradient in the iteration of \ref{eq: AF} with the stochastic gradient, we obtain an iteration of stochastic Amplitude Flow (\ref{eq: SAF}),
\begin{equation}\label{eq: SAF} \tag{SAF}
z^{t+1} = z^t - \mu_t g_\varepsilon(z^t), \quad t \ge 0.
\end{equation}
Note that the indices $r^{t_1}, \ldots, r^{t_K}$ are drawn independently at random for each iteration $t \ge 0$.

Previously, \ref{eq: SAF} was considered for real random matrices $A$, i.e., the entries of $A$ are independent standard Gaussian random variables \cite{Wang.2017, Xiao.2020}. It was shown that if each block $A_r$ corresponds precisely to a single measurement and noise is absent, it is possible to construct an initialization $z^0$ such that the algorithm converges to $x$ at a linear rate \cite[Theorem 1]{Wang.2017}. In \cite[Section 3]{Zhang.2021} the authors showed that the strong convexity of $\q L_0$ in a neighborhood of $x$ is sufficient to guarantee a linear convergence rate of \ref{eq: SAF}, if the algorithm is initialized within this neighborhood. 

Recent results for stochastic gradient descent \cite{Khaled.2020, Orabona.2020, Liu.2022} suggest that it is possible to derive sublinear convergence guarantees for \ref{eq: SAF} for any $A$ analogous to \Cref{thm: AF convergence}. The first statement addresses a constant choice of step sizes.

\begin{theorem}[{Version of \cite[Theorem 2 and Corollary 1]{Khaled.2020}}]\label{thm: SAF convergence constant}
Let $\alpha, \beta$ and $\delta$ be defined as in \eqref{eq: abc params}, $\varepsilon \ge 0$ and fix $\gamma >0$. Consider a sequence $\{z_t \}_{t \ge 0}$ determined by \ref{eq: SAF} with an arbitrary starting point $z^0 \in \bb C^{d}$ and constant step sizes $\mu_t = \mu,$ $t \ge 0$. If the number of iterations $T$ and the step size $\mu$ fulfill
\[
T \ge \frac{4 [ \q L_\varepsilon(z^0)  - \q L_\varepsilon^{\inf}]}{\gamma^2 \mu}
\quad \text{ and } \quad
\mu \le \min \left\{ \frac{1}{\sqrt{\alpha T} \norm{A} }, \frac{1}{ \beta \norm{A}^2}, \frac{\gamma^2}{2 \delta \norm{A}^2} \right\},
\]
then the expected norms of the gradients satisfy
\[
\min_{t = 0, \ldots, T-1} \bb E \norm{\nabla_z \q L_\varepsilon(z^t)}_2 \le \gamma.
\]
Furthermore, if $\mu$ is chosen precisely as the minimum above, it is sufficient that the number of iterations admits
\[
T \ge \max \left\{ \frac{16 \alpha \norm{A}^2 [ \q L_\varepsilon(z^0)  - \q L_\varepsilon^{\inf}]^2  }{\gamma^4}, \frac{4 \beta \norm{A}^2 [ \q L_\varepsilon(z^0)  - \q L_\varepsilon^{\inf}]}{\gamma^2}, \frac{8 \delta \norm{A}^2 [ \q L_\varepsilon(z^0)  - \q L_\varepsilon^{\inf}] }{\gamma^2} \right\}.
\]

\end{theorem}

Comparing \Cref{thm: SAF convergence constant} and \Cref{thm: AF convergence}, we observe that for \ref{eq: SAF} the norms of the gradients are not guaranteed to vanish, unlike for \ref{eq: AF}. Furthermore, the number of iterations required to achieve the upper bound $\gamma$ on the minimal expected norm of the gradients scales in $\gamma^{-2} [ \q L_\varepsilon(z^0)  - \q L_\varepsilon^{\inf}]$ quadratically for \ref{eq: SAF} if $\alpha >0$, while for \ref{eq: AF} this dependency is only linear. In fact, \Cref{thm: SAF convergence constant} generalizes \Cref{thm: AF convergence} as \ref{eq: AF} can be seen as a special case of \ref{eq: SAF} with $\alpha = \delta = 0$ and $\beta = 1$.  

\begin{remark}\label{rem: abs constants}
The choice of $T$ and $\mu$ in \Cref{thm: SAF convergence constant} depends on the quality of the initial guess $\q L_\varepsilon(z^0)  - \q L_\varepsilon^{\inf}$. While it is unknown, it can be estimated from above by $\q L_\varepsilon(z^0)$. Furthermore, an alternative statement can be derived, where $\q L_\varepsilon(z^0)  - \q L_\varepsilon^{\inf}$ is replaced with $\q L_\varepsilon(z^0)$ and the cases with parameter $\delta$ are not present.  
\end{remark}

The main drawback of \Cref{thm: SAF convergence constant} is that the choice of the step size depends on the number of iterations. Hence, this theorem only covers a finite number of iterations and cannot be extended to a convergence result. Moreover, it is known that a small constant step size causes the mean distance between a global maximum and the iterates to grow \cite[Theorem 3.4]{Ding.2019}. If instead the step sizes are decaying, an alternative result is applicable for \ref{eq: SAF}. 

\begin{theorem}[{Version of \cite[Theorem 4]{Liu.2022}}]\label{thm: SAF convergence decreasing expectation}
Let $\varepsilon \ge 0$ and $\beta$ be defined as in \eqref{eq: abc params}. Consider a sequence $\{z_t \}_{t \ge 0}$ determined by \ref{eq: SAF} with an arbitrary starting point $z^0 \in \bb C^{d}$. If the step sizes are positive and satisfy 
\begin{equation}\label{eq: step conditions}
\sum_{t = 0}^\infty \mu_t  = \infty, \quad \sum_{t = 0}^\infty \mu_t^2 < \infty \quad  
\text{ and } \beta \norm{A}^2 \mu_t \le 1,  
\end{equation}
then $\bb E \q L_\varepsilon (z^t)$ converges, $c^2 := \sum_{t =0}^\infty \bb E \mu_t \norm{\nabla_z \q L_\varepsilon(z^t)}_2^2 < \infty$,  
and 
\[
\min_{t = 0, \ldots, T-1} \bb E \norm{\nabla_z \q L_\varepsilon(z^t)}_2 \le c \left[ \sum_{t = 0}^{T-1} \mu_t \right]^{-\frac{1}{2}}
\text{ so that }  
\min_{t = 0, \ldots, T-1} \bb E \norm{\nabla_z \q L_\varepsilon(z^t)}_2 \to 0, \text{ as } t \to \infty.
\]
If additionally $\mu_t$ is decreasing and 
\begin{equation}\label{eq: step conditions 2}
\sum_{t = 0}^\infty  \frac{\mu_t}{\sum_{s = 0}^{t-1} \mu_s}  = \infty,
\end{equation}
the asymptotic rate of convergence is, 
\begin{align*}
\min_{t = 0, \ldots, T-1} \bb E \norm{\nabla_z \q L_\varepsilon(z^t)}_2 & = o \left( \left[ \sum_{t = 0}^{T-1} \mu_t \right]^{-\frac{1}{2}} \right), \text{ as } T \to \infty.
\end{align*}
\end{theorem}

An example of step sizes satisfying conditions \eqref{eq: step conditions} and \eqref{eq: step conditions 2} are provided in the next corollary.
\begin{corollary}[{Version of \cite[Theorem 4]{Liu.2022}}]\label{cor: step size}
Let $\mu_t = \tfrac{\mu}{(1 + t)^{1/2 + \theta} },$ $t \ge 0$, with $\mu>0$ such that $\beta \norm{A}^2 \mu \le 1$ and $0 < \theta < 1/2$, then the iterates of 
\ref{eq: SAF} admit
\[
\min_{t = 0, \ldots, T-1} \bb E \norm{\nabla_z \q L_\varepsilon(z^t)}_2^2 \le \frac{(1/2 - \theta)c^2}{\mu [ (1+ T)^{1/2 - \theta} - 1] },  
\]  
and
\[
\min_{t = 0, \ldots, T-1} \bb E \norm{\nabla_z \q L_\varepsilon(z^t)}_2 = o \left( T^{\theta/2 -1/4}  \right), \ T \to \infty,
\]
where $c$ is the constant defined in \Cref{thm: SAF convergence decreasing expectation}.
Consequently, if the number of iterations admits
\[
T \ge [c^2 \mu^{-1} \gamma^{-2} (1/2 - \theta) + 1 ]^{\frac{2}{1 -2\theta}} - 1. 
\]
for an arbitrary $\gamma >0$, we have $\min_{t = 0, \ldots, T-1} \bb E \norm{\nabla_z \q L_\varepsilon(z^t)}_2 \le \gamma$.
\end{corollary}

As stated in \Cref{thm: SAF convergence decreasing}, for vanishing step sizes the convergence rate is slightly slower compared to the constant step sizes. This can be seen in \Cref{cor: step size} as the dependency of $T$ on $\gamma^{-1}$ is higher than the quartic dependency observed in \Cref{thm: SAF convergence constant}.

The result of \Cref{thm: SAF convergence decreasing expectation} can be strengthened to an almost sure convergence.

\begin{theorem}[{Version of \cite[Theorem 4]{Liu.2022}}]\label{thm: SAF convergence decreasing}
Let $\varepsilon \ge 0$ and $\beta$ be defined as in \eqref{eq: abc params}. Consider a sequence $\{z_t \}_{t \ge 0}$ determined by \ref{eq: SAF} with an arbitrary starting point $z^0 \in \bb C^{d}$. If the step sizes are positive numbers satisfying  \eqref{eq: step conditions}, then $\q L_\varepsilon (z^t)$ converges a.s., the random variable $C^2 := \sum_{t =0}^\infty \mu_t \norm{\nabla_z \q L_\varepsilon(z^t)}_2^2$ is finite a.s., and 
\[
\min_{t = 0, \ldots, T-1} \norm{\nabla_z \q L_\varepsilon(z^t)}_2 \le C \left[ \sum_{t = 0}^{T-1} \mu_t \right]^{-\frac{1}{2}}, \quad \text{a.s.}
\]

If additionally $\mu_t$ is decreasing and admits \eqref{eq: step conditions 2}, then 
\begin{align*}
\min_{t = 0, \ldots, T-1} \norm{\nabla_z \q L_\varepsilon(z^t)}_2 & = o \left( \left[ \sum_{t = 0}^{T-1} \mu_t \right]^{-\frac{1}{2}} \right), \text{ as } T \to \infty, \quad \text{a.s.}
\end{align*}
\end{theorem}


In fact, all statements of \Cref{thm: SAF convergence decreasing expectation} except the convergence of $\bb E \q L_{\varepsilon}(z^t)$ are a consequence of the monotone convergence theorem \cite[Theorem 2.3.4]{Athreya.2006} applied to \Cref{thm: SAF convergence decreasing}.
Since \Cref{thm: SAF convergence decreasing} is analogous to \Cref{thm: SAF convergence decreasing expectation}, a version of \Cref{cor: step size} with an almost sure convergence also holds.

The outcome of \Cref{thm: SAF convergence decreasing} is the following. Firstly, $\{ \q L_\varepsilon(z^t) \}_{t \ge 0}$ converges to some value. Secondly, the obtained in \Cref{thm: SAF convergence decreasing} limit $\inf_{t \ge 0} \norm{\nabla_z \q L_\varepsilon(z^t)}_2 = 0$ implies an existence of $\{ z^{t_k} \}_{k \ge 0}$ such that $\norm{\nabla_z \q L_\varepsilon(z^{t_k})}_2 \to 0$ as $k \to \infty$.

\begin{corollary}\label{col: min and liminf equivalence}
Under conditions of \Cref{thm: SAF convergence decreasing}, we have $\displaystyle\liminf_{t \to \infty} \norm{\nabla_z \q L_\varepsilon(z^{t})}_2 =0$ a.s.
\end{corollary}

This result can be strengthened to the convergence of the gradient almost surely for loss functions with Lipschitz-continuous gradients.

\begin{theorem}[{Version of \cite[Theorem 2]{Orabona.2020}}]\label{thm: a.s. convergence of gradient}
Let $\varepsilon >0$. Consider a sequence $\{z_t \}_{t \ge 0}$ determined by \ref{eq: SAF} with an arbitrary starting point $z^0 \in \bb C^{d}$. If the step sizes are positive numbers satisfying  \eqref{eq: step conditions}, then,
\[
\lim_{t \to \infty} \norm{\nabla_z \q L_\varepsilon(z^t)}_2 = 0, \quad \text{a. s.}
\]
\end{theorem}

We note that there are versions of Amplitude Flow and Stochastic Amplitude Flow, which include a momentum term \cite{Xu.2018, Xian.2022} or additional regularization terms \cite{Filbir.2022, Gao.2022}. Their convergence properties can be established in a similar manner based on \cite{Liu.2022}.  

\begin{remark}\label{eq: Polyak-Lojasiewicz}
The last remark regarding convergence analysis of \ref{eq: SAF} is about the Polyak-{\L}ojasiewicz condition, a conceptual extension of strong convexity for nonconvex functions. Often \cite{Khaled.2020,Patel.2022}, convergence of gradient descent is studied under the Polyak-{\L}ojasiewicz condition
\[
\tfrac{1}{2} \norm{\nabla_z \q L_\varepsilon(z)}_2^2 \ge \eta (\q L_\varepsilon(z)  - \q L_\varepsilon^{\inf} ) \quad \text{ for all } z \in \bb C^d,
\]
with $\eta >0$. In most cases, this condition does not hold for $\q L_\varepsilon$. Consider $z=(0, \ldots, 0)^T$, then we have
\[ 
\norm{\nabla_z \q L_\varepsilon(z)}_2^2 
= \norm{A^*(A z - \sqrt{y + \varepsilon} \circ A z / \sqrt{ |Az|^2 + \varepsilon}  )}_2^2
= 0,
\]
and 
\[
\q L_\varepsilon(z)  - \q L_\varepsilon^{\inf}
= \norm{ \sqrt{|Az|^2 + \varepsilon} - \sqrt{y + \varepsilon} }_2^2  - \q L_\varepsilon^{\inf} 
= \norm{ \sqrt{\varepsilon} - \sqrt{y + \varepsilon} }_2^2  - \q L_\varepsilon^{\inf},
\]  
which is nonzero unless the zero vector is a global minimizer of $\q L_\varepsilon$. In particular, if $z \mapsto Az$ is injective and noise is absent, then the Polyak-{\L}ojasiewicz condition only holds if the true object $x$ is the zero vector. 
\end{remark}

In the following, we address two other algorithms in the literature, which can be seen as special cases of \ref{eq: SAF}. 

\subsection{The Kaczmarz method}\label{sec: kaczmarz}

The first method is the Kaczmarz method for phase retrieval. It was originally designed for solving linear systems \cite{Kaczmarz.1937, Strohmer.2008} and was recently extended to the phase retrieval problem \cite{Wei.2015, Wang.2017, Tan.2019, Romer.2021, Zhang.2021}. The Kaczmarz method is based on the idea that in each iteration a single measurement is selected (randomly) and the object is updated to fit this measurement. To be precise, the iteration of the Kaczmarz method is given by
\begin{equation}\label{eq: Kaczmarz}
z^{t+1} = z^t + \frac{(A_{(r^t)})^*}{\norm{A_{(r^t)}}_2^2} \left[ \sgn( (A z^t)_{r^t} ) y_{r^t} - (A z^t)_{r^t} \right], \quad t \ge 0,
\end{equation}
where $A_{(r)}$ denotes the $r$-th row of the matrix $A$ and index $r^t$ is drawn at random from some distribution $p \in (0,1)^m$. Then, the new iterate satisfies $| |(Az^{t+1})_{r^t}| - \sqrt{y_{r^t}}|= 0$. 

Similarly to \ref{eq: SAF}, the Kaczmarz method for phase retrieval was analyzed for random matrices $A$ \cite{Tan.2019, Zhang.2021}, where its linear convergence was derived. Furthermore, the Kaczmarz method is shown to converge to a neighborhood of $x$ in the noiseless case at a linear rate, however, the radius of this neighborhood is rather large \cite{Wei.2015}.  

In order to present the Kaczmarz iteration as \ref{eq: SAF}, we first introduce a partition $A_r$ of $A$ and corresponding vectors $y^r$. Let $R = m$ and
consider the matrices $A_r \in \bb C^{1 \times d},$ $r \in [m]$ corresponding to a single row $A_{(r)}$. Then, $y^r \in \bb C^1$ is a vector with a sole entry equal to $y_r$. Consequently, by \eqref{eq: generalized gradient}, we have
\[
\nabla_z \q L_{0,r}(z) = A_r^* \left[ A_r z - \sgn( A_r z) \circ \sqrt{y^{r}} \right]
= (A_{(r)})^* \left[ (A z)_{r} - \sgn( (A z)_{r} ) \sqrt{y_{r}} \right].
\]
Thus, iteration \eqref{eq: Kaczmarz} transforms into
\[
z^{t+1}
=z^t - \frac{p_{r_t}}{p_{r_t} \norm{A_{(r^t)}}_2^2} \nabla_z \q L_{0,r_t}(z^t)
= z^t - \mu_t g_0(z^t),
\]
where the stochastic gradient uses a single index $r_t$, i.e., $K=1$, and the step size is set to $\mu_t = p_{r^t}/ \norm{A_{(r^t)}}_2^2$.  
 
The choice of probabilities in the literature is either uniform, $p_r =1/m$, or proportional to $\norm{A_{(r^t)}}_2^2$. The latter choice is sometimes called variance-reducing probabilities, as they minimize the bound on the norm of the gradient provided by \Cref{p: gradient properties}. More precisely, if $p_r = \norm{A_{(r^t)}}_2^2 / \norm{A}_F^2$ we have
\begin{equation}\label{eq: alpha kaczmarz}
\alpha = \max_{r \in [m]} \frac{\norm{A}_F^2 \norm{A_r}^2 }{\norm{A_{(r)}}_2^2}
= \norm{A}_F^2.
\end{equation}
This choice of probabilities also leads to a constant step size $\mu_t = 1 / \norm{A}_F^2$. However, \Cref{thm: SAF convergence constant} only provides a weak upper bound on the minimal expected norm of the gradient.
\begin{corollary}\label{col: Kaczmarz}
Assume that $\q L_0^{\inf} = 0$. Consider a sequence $\{z_t \}_{t=0,\ldots, T}$ determined by \eqref{eq: Kaczmarz} with an arbitrary starting point $z^0 \in \bb C^{d}$ and probabilities $p_r = \norm{A_{(r)}}_2^2 / \norm{A}_F^2,$ $r \in [m]$. Then, for 
\[
T = \left \lceil \norm{A}_F^2 / 4 \norm{A}^2 \right \rceil,
\]
the expected norms of the gradients satisfy
\[
\min_{t = 0, \ldots, T-1} \bb E \norm{\nabla_z \q L_\varepsilon(z^t)}_2 \le 4 \norm{A} \sqrt{ \q L_0(z^0) }.
\]
\end{corollary}
Note that assumption $\q L_0^{\inf} = 0$ is artificially introduced for simplicity of the resulting statement and, in view of \Cref{rem: abs constants}, can be lifted. Yet, the bound $4 \norm{A} \sqrt{ \q L_0(z^0) }$ on the minimal norm of the gradients is rather large and does not imply that any of the iterates is close to a critical point. While we are not able to provide a better convergence analysis for $\mu_t = 1/\norm{A}_F^2$, a version of the Kaczmarz method with a smaller constant or decaying $\mu_t$ can be considered. In this case, the results of \Cref{thm: SAF convergence constant} and \Cref{thm: SAF convergence decreasing} provide stronger convergence guarantees. 

\subsection{STFT phase retrieval and Ptychographic Iterative Engine}\label{sec: PIE}

The second algorithm we would like to consider is the so-called Ptychographic Iterative Engine (\ref{eq: PIE}) \cite{Rodenburg.2004}. It was devised specifically for the phase retrieval problem with the (subsampled) short-time Fourier transform (STFT) matrix $A$. Let $w \in \bb C^d$ and consider measurements of the form
\begin{equation}\label{eq: ptycho meas}
y^r = |F[ S_{s_r} w \circ  x]|^2  + n, \quad r \in [R],
\end{equation}
where $F$ is the discrete Fourier transform \eqref{eq: DFT}, $S_s,$ $s \in [d]$ denote circular shift operators \eqref{eq: shift op} and $s_1, \ldots, s_R \in [d]$ denote $1 \le R \le d$ unique shift positions of the vector $w$. 

The measurements \eqref{eq: ptycho meas} arise in ptychography \cite{Pfeiffer.2018} and acoustic applications \cite{Gerkmann.2015}. In fact, they can be seen as an instance of the phase retrieval problem \eqref{eq: meas}. For this, we combine all measurements as a vector $y = ( (y^1)^T, \ldots, (y^R)^T )^{T}$ and construct a row-block measurement matrix $A$ as
\begin{equation}\label{eq: STFT matrix}
A := 
\begin{bmatrix}
A_1 \\
\vdots\\
A_R 
\end{bmatrix}
:=
\begin{bmatrix}
F \diag(S_{s_1} w) \\
\vdots \\
F \diag(S_{s_R} w) 
\end{bmatrix}.
\end{equation}
If $\{ s_1, \ldots, s_R \} = [d]$, the matrix $A$ corresponds to the discrete STFT and in all other cases it is a subsampled STFT matrix. 

One of the methods for recovery from STFT measurements \eqref{eq: ptycho meas} is Ptychographic Iterative Engine. Given an initial guess $z^0$ it constructs the $t$-th iterate $z^{t}$, $t \ge 0$, by performing the following steps. \\
\medskip
\begin{algorithm}[H]
\caption{PIE iteration, version of \cite{Rodenburg.2004, Maiden.2009}}
\label{alg: PIE iteration}
\SetAlgoLined
\SetKwInOut{Input}{Input}
\SetKwInOut{Output}{Output}
\Input{Ptychographic measurements $y^r$, $r \in [R]$ as in \eqref{eq: ptycho meas}, previous object iterate $z^t \in \bb C^d$, parameter $\alpha_t >0$. }
\Output{$z^{t+1} \in \bb C^d$.}
1. Select a shift position $r^t \in [R]$ and denote $s_{r^t}$ by $s_t$.\\
2. Construct an exit wave $\psi = S_{-s_{t}} z^t \circ w$.\\
3. Compute its Fourier transform $\Psi = F \psi$.\\
4. Correct the magnitudes of $\Psi$ as $\Psi' = \sqrt{ y^{r^t} } \circ \sgn \Psi$.\\
5. Find an exit wave $\psi'$ corresponding to $\Psi'$ via $\psi' = F^{-1} \Psi' $. \\
6. Return $z^{t+1} = z^t + \frac{\alpha_t}{\norm{w}_\infty^2}  \diag( \conj{ S_{s_{t}} w}) [\psi' - \psi]$.\\
\end{algorithm}
\medskip

In the literature, two ways of choosing the shift positions $r^t$ are considered. Originally, in \cite{Rodenburg.2004}, the shift $r^t$ was selected such that regions corresponding to $r^{t-1}$ and $r^t$ overlap. Later, in \cite{Maiden.2009, Maiden.2017} indices $r^t$ are looping through the set $[R]$, which is randomly shuffled every loop. The parameter $\alpha_t$ is generally set to a constant value, e.g., $\alpha_t = 0.05$.


There are several interpretations of the \ref{eq: PIE} iteration. The first states that the iteration of \ref{eq: PIE} computes the measurements corresponding to the current shift position $r$, and adjusts the magnitudes if they do not agree with the measurements $y^{r}$. Then, the algorithm moves from the previous position $z^{t}$ in the direction of the object corresponding to the corrected measurements, which gives the new iterate $z^{t+1}$. Therefore, $\alpha_t$ plays a role of step size. All steps of \Cref{alg: PIE iteration} combined into one provide the following update rule
\begin{equation}\label{eq: PIE}\tag{PIE}
z^{t+1} = z^{t} + \frac{\alpha_t  \diag(\conj{S_{s_{t}} w}) }{\norm{w}_\infty^2} \left[ F^{-1} \diag\left( \frac{ \sqrt{ y^{r^t} }}{ |F [S_{s_{t}} w \circ z^t]|} \right) F   - I_d\right] (S_{s_{t}} w \circ z^t).
\end{equation}

In \cite{Thibault.2013}, the authors note that in this form the new iterate of \ref{eq: PIE} is the global minimizer of
\[
z^{t+1} = \argmin_{z \in \bb C^d}  \norm{ S_{s_{t}} w \circ z -  \psi'}_2^2 + \norm{ \left[ \frac{\norm{w}_\infty^2}{\alpha_t} I_d - \diag(|S_{s_{t}} w|^2) \right] (z - z^{t}) }_2^2.
\]  
where $\psi'$ is a result of step 5 in \Cref{alg: PIE iteration}.This is reminiscent of the proximal mapping \cite[Chapter 6]{Beck.2017}. Alternatively, the iteration of the \ref{eq: PIE} algorithm is the gradient descent step \cite{Maiden.2017} for the loss function
\[
\norm{ S_{s_{t}} w \circ z -  \psi'}_2^2
\]  
with the step sizes $\mu_t = \alpha_t / \norm{w}_\infty^2$. Based on the concepts from \Cref{sec: SAF}, we are able to show that \ref{eq: PIE} is an instance of \ref{eq: SAF}.
\begin{theorem}\label{thm: PIE as stochastic gradient}
Let $A$ be the measurement matrix \eqref{eq: STFT matrix} corresponding to the ptychographic measurements \eqref{eq: ptycho meas}. The iteration of \ref{eq: PIE} is the iteration of \ref{eq: SAF} with $\varepsilon = 0$ and step size 
\[
\mu_t  = \frac{\alpha_t p_{r^t} }{d \norm{w}_\infty^2}
= \frac{\alpha_t p_{r^t}}{\norm{A_{r^t}}^2}.
\] 
The stochastic gradient $g_0$ of $\q L_0$ is constructed based on the partition $A_1, \ldots, A_R$ as in \eqref{eq: STFT matrix}. It samples $K=1$ index from a distribution $p \in \bb R^R$, where $p_r$ is a probability to chose index $r$ as $r^t$ in the iteration of \ref{eq: PIE}, $r \in [R]$.
\end{theorem}

Since the norms $\norm{A_{r}}$ are all equal, the uniform sampling $p_r = 1/R$ will minimize $\alpha$ in \eqref{eq: abc params} for \ref{eq: PIE}. Then, the stochastic gradient descent representation of the \ref{eq: PIE} algorithm allows us to derive its convergence if $\alpha_t$ decays with $t$.  

\begin{corollary}\label{col: PIE convergence}
Consider a sequence $\{ z^t \}_{t \ge 0}$ determined by \ref{eq: PIE} with an arbitrary starting point $z^0 \in \bb C^d$ and uniform probabilities $p_r = 1/R$. Let $\alpha_t = \alpha / (1+t)^{1/2 + \theta}$ for some $\alpha >0$ and $0<\theta <1/2$. Then, the sequence $\{ \q L(z^t) \}_{t \ge 0}$ converges almost surely and we have  
\[
\min_{t = 0, \ldots, T-1} \norm{ \nabla_z \q L(z^{t}) }_2 = o ( T^{-(1/2-\theta)} ) \text{ as } T \to \infty, \ \text{a.s.}
\]
\end{corollary}

Consequently, \Cref{col: PIE convergence} guarantees that eventually the norm of the gradient will become sufficiently small, which is a common stopping criteria for algorithms in practice. Note that stronger convergence guarantees such as \Cref{thm: a.s. convergence of gradient} are applicable for a version of \ref{eq: PIE} where $\q L_0$ is replaced by $\q L_\varepsilon$ with $\varepsilon >0$.

\section{Proofs}\label{sec: proofs}

\subsection{Wirtinger derivatives and proof of descent lemma}\label{sec: Wirtinger derivatives}

In the following, we review basic concepts of Wirtinger derivatives based on \cite{Hunger.2008, Bouboulis.2010}. As the loss function \eqref{eq: loss} is a real-valued function of complex variables, it is not holomorphic \cite[Proposition 4.0.1]{Hunger.2008}. However, for gradient-based optimization of such functions Wirtinger derivatives are extremely useful. 

Let $f: \bb C \to \bb C$ be a function of a complex variable $z = \alpha + i \beta$. We say that a function is differentiable (in the real sense), if it is differentiable with respect to the real part $\alpha$ and the imaginary part $\beta$. The Wirtinger derivatives of $f$ are then defined as
\[
\frac{\partial f}{\partial z} := \frac{1}{2}\frac{\partial f}{\partial \alpha}  - \frac{i}{2}\frac{\partial f}{\partial \beta},
\quad 
\frac{\partial f}{\partial \bar{z} } := \frac{1}{2}\frac{\partial f}{\partial \alpha}  + \frac{i}{2}\frac{\partial f}{\partial \beta}.
\]
Wirtinger derivatives can be seen as an alternative parametrization of a function $f$ with respect to conjugate coordinates $z$ and $\conj z$. Since $\conj z$ is antiholomorphic in $z$, we have that $\tfrac{\partial z}{\partial \conj z}=0$ and vice versa. Thus, $z$ is treated as a constant when differentiating with respect to $\conj z$ and, analogously, $\conj z$ is treated as a constant when differentiating with respect to $z$. 

The definition of Wirtinger derivatives is a linear transformation of real and imaginary derivatives. Hence, all standard results from real analysis such as arithmetical operations and the chain rule remain true for Wirtinger derivatives. Furthermore, the conjugation rule applies
\begin{equation}\label{eq: conj rule}
\overline{\frac{\partial f}{\partial z} } =
\frac{\partial \bar f}{\partial \bar z}
\quad
\text{and}
\quad
\overline{\frac{\partial f}{\partial \bar z} } =
\frac{\partial \bar f}{\partial z}.
\end{equation}

In analogy to real analysis, Wirtinger derivatives for multivariate functions are defined as row vectors of derivatives. That is for $f: \bb C^{d} \to \bb C$ they are given by 
\[
\frac{\partial f}{\partial z} = \left( \frac{\partial f}{\partial z_1}, \ldots,  \frac{\partial f}{\partial z_d} \right)
\quad
\text{and}
\quad
\frac{\partial f}{\partial \bar{z} } = \left( \frac{\partial f}{\partial \bar z_1}, \ldots,  \frac{\partial f}{\partial \bar z_d} \right).
\] 

Turning to real-valued functions $f: \bb C^{d} \to \bb R$, the differential of $f$ can be rewritten as
\[
d f = \sum_{j \in [d]}\left[ \frac{\partial f}{\partial \alpha_j} d \alpha_j + \frac{\partial f}{\partial \beta_j} d \beta_j \right]
= \sum_{j \in [d]}\left[ \frac{\partial f}{\partial z_j} d z_j + \frac{\partial f}{\partial \conj z_j} d \conj z_j \right]
= \frac{\partial f}{\partial z} d z + \frac{\partial f}{\partial \conj z} d \conj z 
= 2 \RE\left[ \frac{\partial f}{\partial z} d z \right],
\]
where \eqref{eq: conj rule} and real-valuedness of $f$ were used in the last step. Consequently, the direction of the steepest ascend $d z$ is aligned with the conjugate transpose of $\frac{\partial f}{\partial z}$, which is why Wirtinger gradient in direction of $z$ is defined as
\[
\nabla_z f := \left(\frac{\partial f}{\partial z} \right)^* = \left(\frac{\partial f}{\partial \conj z} \right)^T.
\]
Note that the second equality once again follows from the conjugation rule \eqref{eq: conj rule} and real-valuedness of $f$. Analogously, the gradient in direction of $\conj z$ is given by $\nabla_{\conj z} f := \left(\frac{\partial f}{\partial \conj z} \right)^* = \conj{\nabla_z f}$. 

For a twice continuously differentiable function $f$ its Wirtinger Hessian matrix is defined as
\begin{equation}\label{eq: Hessian def}
\nabla^2 f = 
{
\renewcommand\arraystretch{1.75}
\begin{bmatrix}
\frac{\partial}{\partial z}  \nabla_z f & \frac{\partial}{\partial \conj z}  \nabla_z f \\
\frac{\partial}{\partial z}  \nabla_{\conj{z}} f & \frac{\partial}{\partial \conj{z}}  \nabla_{\conj{z}} f
\end{bmatrix}
}
= 
\begin{bmatrix}
\nabla_z f \\
\nabla_{\conj{z}} f
\end{bmatrix}
\begin{bmatrix}
\frac{\partial}{\partial z} & \frac{\partial}{\partial \conj z}
\end{bmatrix}
=
\begin{bmatrix}
0 & I_d \\
I_d & 0
\end{bmatrix}
{
\renewcommand\arraystretch{1.25}
\begin{bmatrix}
(\frac{\partial}{\partial z} f)^T \\ 
(\frac{\partial}{\partial \conj z} f)^T
\end{bmatrix}
}
\begin{bmatrix}
\frac{\partial}{\partial z} & \frac{\partial}{\partial \conj z}
\end{bmatrix}.
\end{equation}
Then, for $z,v \in \bb C^d$ the Taylor expansion up to the second order gives
\begin{equation}\label{eq: Taylor}
f(z+v) = f(z) + 
\begin{bmatrix}
\nabla_z f \\
\nabla_z f
\end{bmatrix}^*
\begin{bmatrix}
v \\
\bar v
\end{bmatrix}
+
\begin{bmatrix}
v \\
\bar v
\end{bmatrix}^*
\int_0^1 
(1 - s)
\nabla^2 f(z + s v)
d s
\begin{bmatrix}
v \\
\bar v
\end{bmatrix}.
\end{equation} 
For $\q L_\varepsilon$, the second order term was computed and bounded from above in \cite{Xu.2018}. In addition, we derive a lower bound in the next lemma. 
\begin{lemma}\label{l: Hessian bounds}
Let $\varepsilon >0$. For any $z,u \in \bb C^d$ the Hessian of $\q L_\varepsilon$ satisfies
\[
- 2 (\norm{ y + \varepsilon }_\infty^{1/2} \varepsilon^{-1/2} -1) \norm{A}^2 \norm{u}_2^2
\le 
\begin{bmatrix}
u \\
\bar u
\end{bmatrix}^*
\nabla^2 \q L_\varepsilon(z)
\begin{bmatrix}
u \\
\bar u
\end{bmatrix}
\le 2 \norm{A}^2 \norm{u}_2^2.
\]
\end{lemma}

\begin{proof}
The proof is based on computations in \cite[pp. 26-28]{Xu.2018}, which expand the quadratic term as
\begin{align*}
\begin{bmatrix}
u \\
\bar u
\end{bmatrix}^*
\nabla^2 \q L_\varepsilon(z)
\begin{bmatrix}
u \\
\bar u
\end{bmatrix}
& = 2 \sum_{j \in [m]} \left[ 1 - \frac{\varepsilon \sqrt{y_j + \varepsilon}}{( |(Az)_j|^2 + \varepsilon)^{3/2}} \right] |(Au)_j|^2 \\
& \quad + \sum_{j \in [m]} \frac{\sqrt{y_j + \varepsilon}}{( |(Az)_j|^2 + \varepsilon)^{3/2}} \left[  \RE( (Az)_j^2 (\conj{Au})_j^2 ) - |(Az)_j|^2 |(Au)_j|^2 \right].
\end{align*}
The upper bound was derived in \cite[pp. 26-28]{Xu.2018} by applying $\RE(\alpha) \le |\alpha|$ for $\alpha \in \bb C$ and noticing that the fraction in the first sum is nonnegative. For the lower bound, we use that $\RE(\alpha) \ge -|\alpha|$ for $\alpha \in \bb C$, which yields
\begin{align*}
\begin{bmatrix}
u \\
\bar u
\end{bmatrix}^*
\nabla^2 \q L_\varepsilon(z)
\begin{bmatrix}
u \\
\bar u
\end{bmatrix}
& 
\ge 2 \sum_{j \in [m]} \left[ 1 - \frac{ (|(Az)_j|^2 + \varepsilon) \sqrt{y_j + \varepsilon}}{( |(Az)_j|^2 + \varepsilon)^{3/2}} \right] |(Au)_j|^2 \\
& = 2 \sum_{j \in [m]} \left[ 1 - \frac{ \sqrt{y_j + \varepsilon}}{ \sqrt{ |(Az)_j|^2 + \varepsilon } } \right] |(Au)_j|^2 \\
& \ge 2 \sum_{j \in [m]} \left[ 1 - \frac{ \sqrt{y_j + \varepsilon}}{ \sqrt{ 0 + \varepsilon } } \right] |(Au)_j|^2 
= - 2 \sum_{j \in [m]} \left[ \frac{ \sqrt{y_j + \varepsilon}}{ \sqrt{\varepsilon } } - 1 \right] |(Au)_j|^2. 
\end{align*}
Note that $\sqrt{y_j +\varepsilon}/ \sqrt \varepsilon \ge 1$ for all $j \in [m]$. Thus, the multipliers are nonnegative. Consequently, we bound them by the largest one, which gives
\begin{align*}
\begin{bmatrix}
u \\
\bar u
\end{bmatrix}^*
\nabla^2 \q L_\varepsilon(z)
\begin{bmatrix}
u \\
\bar u
\end{bmatrix}
& \ge - 2  (\norm{ y + \varepsilon }_\infty^{1/2} \varepsilon^{-1/2} -1)  \sum_{j \in [m]} |(Au)_j|^2 \\
& \ge - 2  (\norm{ y + \varepsilon }_\infty^{1/2} \varepsilon^{-1/2} -1) \norm{A u}_2^2 \\ 
& \ge - 2  (\norm{ y + \varepsilon }_\infty^{1/2} \varepsilon^{-1/2} -1) \norm{A}^2 \norm{u}_2^2.
\end{align*}
\end{proof}

With these definitions and preliminary results, we are ready to prove \Cref{l: descent lemma}.

\begin{proof}[Proof of \Cref{l: descent lemma}]
The first part was derived in \cite{Xu.2018} by combining the Taylor expansion up to the second order with the upper bound in \Cref{l: Hessian bounds}. More precisely, let $\varepsilon >0$. Then, by \eqref{eq: Taylor} and $\nabla_{\conj{z}} \q L_\varepsilon = \conj{\nabla_z \q L_\varepsilon}$ we have
\begin{align*}
\q L_\varepsilon(z+v) 
& = \q L_\varepsilon(z) + 
\begin{bmatrix}
\nabla_z \q L_\varepsilon \\
\overline{\nabla_z \q L_\varepsilon}
\end{bmatrix}^*
\begin{bmatrix}
v \\
\bar v
\end{bmatrix}
+
\begin{bmatrix}
v \\
\bar v
\end{bmatrix}^*
\int_0^1 
(1 - s)
\nabla^2 \q L_\varepsilon(z + s v)
d s
\begin{bmatrix}
v \\
\bar v
\end{bmatrix}
\\
& \le \q L_\varepsilon(z) + 2 \RE ( \langle \nabla_z \q L_\varepsilon ,v \rangle) + 2 \norm{A}^2 \norm{v}_2^2 \int_0^1 (1 - s)d s \\
& = \q L_\varepsilon(z) + 2 \RE (\langle \nabla_z \q L_\varepsilon ,v \rangle) + \norm{A}^2 \norm{v}_2^2.
\end{align*}
The case $\varepsilon = 0$ follows by taking limit $\varepsilon \to 0+$ in the obtained inequality.

For the Lipschitz continuity of $\q L_\varepsilon$, we first consider $\q L_\varepsilon$ as a function of real vector $(\alpha, \beta) = (\RE z, \IM z) \in \bb R^{2d}$. If $\q L_\varepsilon$ admits a global upper bound on the spectral norm of the Hessian matrix in terms of real and imaginary derivatives,
\[
\norm{
\left(
{
\renewcommand\arraystretch{1.25}
\begin{bmatrix}
(\frac{\partial}{\partial \alpha} \q L_\varepsilon)^T \\ 
(\frac{\partial}{\partial \beta} \q L_\varepsilon)^T
\end{bmatrix}
}
\begin{bmatrix}
\frac{\partial}{\partial \alpha} & \frac{\partial}{\partial \beta}
\end{bmatrix}
\right)(z)
}
\le \tilde L, \quad \text{for all } z \in \bb C^d,
\]
then \cite[Theorem 5.12]{Beck.2017} yields that its gradient (in the real sense) is Lipschitz continuous with constant $\tilde L$, that is  
\begin{equation}\label{eq: Lipschitz in real sense}
\norm{ 
{
\renewcommand\arraystretch{1.25}
\begin{bmatrix} 
(\frac{\partial \q L_\varepsilon}{\partial \alpha})^T(z) \\
(\frac{\partial \q L_\varepsilon}{\partial \beta})^T(z)  
\end{bmatrix}
-
\begin{bmatrix} 
(\frac{\partial \q L_\varepsilon}{\partial \alpha})^T(v) \\
(\frac{\partial \q L_\varepsilon}{\partial \beta})^T(v)  
\end{bmatrix}
}
}_2
\le 
\tilde L 
\norm{
\begin{bmatrix} 
\RE(z-v) \\
\IM(z-v)  
\end{bmatrix}
}_2.
\end{equation}
Recall that Wirtinger derivatives and derivatives with respect to real and imaginary components are related as
\[
\begin{bmatrix}
\frac{\partial}{\partial z} & \frac{\partial}{\partial \conj z}
\end{bmatrix}
= 
\begin{bmatrix}
\frac{\partial}{\partial \alpha} & \frac{\partial}{\partial \beta}
\end{bmatrix}
\begin{bmatrix}
1/2 I_d & 1/2 I_d \\
-i/2 I_d & i/2 I_d \\
\end{bmatrix}
=: 
\begin{bmatrix}
\frac{\partial}{\partial \alpha} & \frac{\partial}{\partial \beta}
\end{bmatrix}
U.
\]
The matrix $\sqrt 2 U$ is unitary, so that $2 U^* U = 2 U U^* = I_{2d}$ and $\norm{\sqrt 2 U q}_2 = \norm{q}_2$ for any $q \in \bb C^{2d}$. Moreover, we have  
\[
U^*
=
\begin{bmatrix}
0 & I_d \\
I_d & 0
\end{bmatrix} 
U^T
\quad \text{ and }
\quad 
2U^*
\begin{bmatrix}
\RE v \\ 
\IM v
\end{bmatrix}
=
\begin{bmatrix}
v \\ 
\conj v
\end{bmatrix}.
\]
Since the Hessian with real and imaginary derivatives is a symmetric matrix, we can rewrite its spectral norm as a maximum of absolute values of inner products, 
\begin{align*}
& 
\norm{
\left(
{
\renewcommand\arraystretch{1.25}
\begin{bmatrix}
(\frac{\partial}{\partial \alpha} \q L_\varepsilon)^T \\ 
(\frac{\partial}{\partial \beta} \q L_\varepsilon)^T
\end{bmatrix}
}
\begin{bmatrix}
\frac{\partial}{\partial \alpha} & \frac{\partial}{\partial \beta}
\end{bmatrix}
\right)(z)
}
= 
\max_{v \in \bb C^d, \norm{v}_2 = 1}
\left|
\begin{bmatrix}
\RE v \\ 
\IM v
\end{bmatrix}^T
{
\renewcommand\arraystretch{1.25}
\begin{bmatrix}
(\frac{\partial}{\partial \alpha} \q L_\varepsilon)^T \\ 
(\frac{\partial}{\partial \beta} \q L_\varepsilon)^T
\end{bmatrix}
}
\begin{bmatrix}
\frac{\partial}{\partial \alpha} & \frac{\partial}{\partial \beta}
\end{bmatrix}
\begin{bmatrix}
\RE v \\ 
\IM v
\end{bmatrix}
\right| 
(z) \\
& \quad 
= 
\max_{v \in \bb C^d, \norm{v}_2 = 1}
\left|
\begin{bmatrix}
\RE v \\ 
\IM v
\end{bmatrix}^T
2 U
\begin{bmatrix}
0 & I_d \\
I_d & 0
\end{bmatrix}
U^T
{
\renewcommand\arraystretch{1.25}
\begin{bmatrix}
(\frac{\partial}{\partial \alpha} \q L_\varepsilon)^T \\ 
(\frac{\partial}{\partial \beta} \q L_\varepsilon)^T
\end{bmatrix}
}
\begin{bmatrix}
\frac{\partial}{\partial \alpha} & \frac{\partial}{\partial \beta}
\end{bmatrix}
2U U^*
\begin{bmatrix}
\RE v \\ 
\IM v
\end{bmatrix}
\right| 
(z) \\
& \quad =
\max_{v \in \bb C^d, \norm{v}_2 = 1}
\left|
\begin{bmatrix}
v \\ 
\conj v
\end{bmatrix}^*
\begin{bmatrix}
0 & I_d \\
I_d & 0
\end{bmatrix}
{
\renewcommand\arraystretch{1.25}
\begin{bmatrix}
(\frac{\partial}{\partial z} \q L_\varepsilon)^T \\ 
(\frac{\partial}{\partial \conj z} \q L_\varepsilon)^T
\end{bmatrix}
}
\begin{bmatrix}
\frac{\partial}{\partial z} & \frac{\partial}{\partial \conj z}
\end{bmatrix}
\begin{bmatrix}
v \\ 
\conj v
\end{bmatrix}
\right| 
(z) \\
& \quad =
\max_{v \in \bb C^d, \norm{v}_2 = 1}
\left|
\begin{bmatrix}
v \\ 
\conj v
\end{bmatrix}^*
\nabla^2 \q L_\varepsilon (z) 
\begin{bmatrix}
v \\ 
\conj v
\end{bmatrix}
\right| 
\le \max_{v \in \bb C^d, \norm{v}_2 = 1} 2 L \norm{v}_2^2 = 2 L, 
\end{align*}  
where in transition to the last line we used \eqref{eq: Hessian def} and the inequality follows from \Cref{l: Hessian bounds} and the definition \eqref{eq: Lipschitz constant} of $L$. Consequently, $\tilde L$ in \eqref{eq: Lipschitz in real sense} is bounded above by $2 L$. Finally, we apply \eqref{eq: Lipschitz in real sense} to obtain Lipschitz continuity of the Wirtinger gradient,
\begin{align*}
& \norm{ \nabla_z \q L_\varepsilon(z) - \nabla_z \q L_\varepsilon(v) }_2^2
= \frac{1}{2} \norm{ 
\begin{bmatrix}
\nabla_z \q L_\varepsilon(z)  \\
\conj{ \nabla_z \q L_\varepsilon(z)}
\end{bmatrix}
 - 
\begin{bmatrix}
\nabla_z \q L_\varepsilon(v)  \\
\conj{ \nabla_z \q L_\varepsilon(v) }
\end{bmatrix}
}_2^2 \\
& \quad  = \frac{1}{2} \norm{ 
\begin{bmatrix}
\nabla_z \q L_\varepsilon(z)  \\
\nabla_{\conj{z}} \q L_\varepsilon(z)
\end{bmatrix}
 - 
\begin{bmatrix}
\nabla_z \q L_\varepsilon(v)  \\
\nabla_{\conj{z}} \q L_\varepsilon(v)
\end{bmatrix}
}_2^2
= \frac{1}{2} \norm{ 
{
\renewcommand\arraystretch{1.25}
\begin{bmatrix}
(\frac{\partial \q L_\varepsilon}{\partial z})^T (z)  \\
(\frac{\partial \q L_\varepsilon}{\partial \conj{z} })^T (z)
\end{bmatrix}
 - 
\begin{bmatrix}
(\frac{\partial \q L_\varepsilon}{\partial z})^T (v)  \\
(\frac{\partial \q L_\varepsilon}{\partial \conj{z} })^T (v)
\end{bmatrix}
}
}_2^2 \\
& \quad = \frac{1}{2} \norm{  U^T \left(
{
\renewcommand\arraystretch{1.25}
\begin{bmatrix}
(\frac{\partial \q L_\varepsilon}{\partial \alpha})^T (z)  \\
(\frac{\partial \q L_\varepsilon}{\partial \beta })^T (z)
\end{bmatrix}
 - 
\begin{bmatrix}
(\frac{\partial \q L_\varepsilon}{\partial \alpha})^T (v)  \\
(\frac{\partial \q L_\varepsilon}{\partial \beta })^T (v)
\end{bmatrix}
}
\right)
}_2^2
= \frac{1}{4} \norm{ \left(
{
\renewcommand\arraystretch{1.25}
\begin{bmatrix}
(\frac{\partial \q L_\varepsilon}{\partial \alpha})^T (z)  \\
(\frac{\partial \q L_\varepsilon}{\partial \beta })^T (z)
\end{bmatrix}
 - 
\begin{bmatrix}
(\frac{\partial \q L_\varepsilon}{\partial \alpha})^T (v)  \\
(\frac{\partial \q L_\varepsilon}{\partial \beta })^T (v)
\end{bmatrix}
}
\right)
}_2^2 \\
& \quad \le \frac{4 L^2}{4} \norm{
\begin{bmatrix} 
\RE(z-v) \\
\IM(z-v)  
\end{bmatrix}
}_2^2
=L^2 \norm{z-v}_2^2.
\end{align*}
Note that for $\varepsilon = 0$, the generalized gradient $\nabla_z \q L_0$ is not continuous due to the discontinuity of the sign function. Therefore, the Lipschitz continuity only applies for $\varepsilon >0$ and the constant $L$ deteriorates as $\varepsilon$ decreases. 
\end{proof}

\subsection{Proofs for Stochastic Amplitude Flow}\label{sec: complex sgd}

In this subsection, we provide the proofs for \Cref{sec: SAF}. While they are close to the referenced result, there are extensions or changes. Thus, for completeness and better readability of the paper, we include full proofs. 
Let us start with \Cref{p: gradient properties}.

\begin{proof}[Proof of \Cref{p: gradient properties}.]
By the linearity of expectation we obtain
\begin{align*}
\bb E g_\varepsilon(z) & = \frac{1}{K} \sum_{k \in [K]} \bb E \frac{1}{p_{r_k}} \nabla_z \q L_{\varepsilon,r_k}(z)
= \frac{1}{K} \sum_{k \in [K]} \sum_{r \in [R]} \frac{p_r}{p_r} \nabla_z \q L_{\varepsilon,r}(z) 
= \nabla_z \q L_{\varepsilon}(z) .
\end{align*}
For the expectation of the norm squared, we use equation (41) in \cite{Khaled.2020}, which gives
\[
\bb E \norm{g_\varepsilon(z)}_2^2 = \left[1 - \frac{1}{K} \right] \norm{ \nabla_z \q L_{\varepsilon}(z) }_2^2  + \sum_{r \in [R]} \frac{1}{K p_r}\norm{ \nabla_z \q L_{\varepsilon,r}(z) }_2^2.
\]
An application of \Cref{thm: AF convergence} for $\q L_{\varepsilon, r}$ with $\mu = 1/\norm{A_r}^{2}$ and $z_+ = z - \mu \nabla_z \q L_{\varepsilon,r}(z)$ gives 
\[
\q L_{\varepsilon,r}(z_+) \le \q L_{\varepsilon,r}(z) - \norm{A_r}^{-2} \norm{\nabla_z \q L_{\varepsilon,r}(z)}_2^2,
\]
which is equivalent to
\[
\norm{\nabla_z \q L_{\varepsilon,r}(z)}_2^2 \le \norm{A_r}^{2} \left[\q L_{\varepsilon,r}(z) - \q L_{\varepsilon,r}(z_+) \right] 
\le \norm{A_r}^{2} \left[\q L_{\varepsilon,r}(z) - \q L_{\varepsilon,r}^{\inf} \right].
\]
Substituting this bound into the expectation gives
\begin{align}
\bb E \norm{g_\varepsilon(z)}_2^2 & \le \beta \norm{ \nabla_z \q L_{\varepsilon}(z) }_2^2  + \sum_{r \in [R]} \frac{\norm{A_r}^{2}}{K p_r} \left[\q L_{\varepsilon,r}(z) - \q L_{\varepsilon,r}^{\inf} \right] \nonumber \\
& \le \beta \norm{ \nabla_z \q L_{\varepsilon}(z) }_2^2  + \max_{r \in [R]} \frac{\norm{A_r}^{2}}{K p_r} \left[ \q L_{\varepsilon}(z)  - \q L_{\varepsilon}^{\inf} + \q L_{\varepsilon}^{\inf} - \sum_{r \in [R]} \q L_{\varepsilon,r}^{\inf} \right]  \nonumber \\
& = \alpha [\q L_{\varepsilon}(z)  - \q L_{\varepsilon}^{\inf}] + \beta \norm{ \nabla_z \q L_{\varepsilon}(z) }_2^2  + \delta, \label{eq: norm expectation bound}
\end{align}
where we used definitions \eqref{eq: abc params}. 
\end{proof}

The starting point of the convergence analysis of \ref{eq: SAF} is the reinterpretation of \Cref{l: descent lemma} for stochastic gradients. Since the next proofs will involve tools from the theory of stochastic processes, we first review some of the necessary definitions based on \cite{Athreya.2006}.      

Let $(\q R, \q F, \bb P)$ be a probability space induced by random indices $\{ r^t :=(r^{t_1}, \ldots, r^{t_K}) \}_{t \ge 0}$. Recall that a random variable $X$ is a function from $(\q R, \q F)$ to some measurable space $(\q X, \Sigma)$, which is $\q F$-measurable, i.e., for all sets $B \in \Sigma$ the preimage $X^{-1}(B)$ is in $\q F$. We further say that $X$ is $\q G$-measurable for some $\q G \subseteq \q F$ if for all sets $B \in \Sigma$ we have $X^{-1}(B) \in \q G$. For an index set $\q J \subseteq \bb N$ and random variables $\{ X_j \}_{j \in \q J}$ the sigma algebra $\q G$ is generated by $\{ X_j \}_{j \in \q J}$ is the smallest sigma algebra, which includes all preimages $\{ X_j^{-1}(B),\ B \in \Sigma,  j \in \q J\}$.

A conditional expectation of $X$ with respect to sigma algebra $\q G$ is a $\q G$-measurable random variable $Y = \bb E [\,X\,|\,\q G\,]$ such that $\bb E X \q I_{A} = \bb E Y \q I_{A}$ for all $A \in \q G$, where $\q I_A$ is the indicator function of the set $A$. If $X$ is $\q G$-measurable, then $X = \bb E [\,X\,|\,\q G\,]$. For $\q H \subseteq \q G \subseteq \q F$ conditional expectations admit the tower property 
\[
\bb E [\,X\,|\,\q H\,] 
= \bb E [\,\bb E [\,X\,|\,\q H\,]\,|\,\q G\,] 
= \bb E [\,\bb E [\,X\,|\,\q G\,]\,|\,\q H\,]
\]
and if $\q G$ is the trivial sigma-algebra $\{ \varnothing, \q R\}$, then $\bb E [\,X\,|\,\q G\,] = \bb E X$.

Define $\q F_0 := \{ \varnothing, \q R \}$ as the trivial sigma-algebra and for $t \in \bb N$ let $\q F_t$ be a sigma-algebra generated by $\{ r^s = (r^{s_1}, \ldots, r^{s_K}) \}_{s = 0, \ldots, t-1}$. Then, by construction, $\q F_0 \subseteq \q F_1 \subseteq \ldots \subseteq \q F$ and $\{\q F_t \}_{t \ge 0}$ is called a filtration. A sequence of random variables $\{X_t\}_{t \ge 0}$ is said to be adapted to filtration $\{ \q F_t \}_{t \ge 0}$ is for all $t \ge 0$ $X_t$ is $\q F_t$-measurable. If a sequence $\{X_t\}_{t \ge 0}$ satisfies $X_t = \bb E [\,X_{t+1}\,|\,\q F_t\,]$ it is called a martingale. 

An example of the sequences adapted to $\{\q F_t\}_{t \ge 0}$ are $\{z^t\}_{t \ge 0}$, $\{ \q L_\varepsilon(z^t) \}_{t \ge 0}$ and $\{\nabla_z \q L_\varepsilon(z^t) \}_{t \ge 0}$. This is a consequence of viewing these sequences as functions of all indices $r^s$, $0 \le s \le t-1$, i.e., $z^t = f_t(r^0, \ldots, r^{t-1})$ for some measurable function $f_t$. Note that $\{g_\varepsilon(z^t) \}$ is not adapted to $\{\q F_t\}_{t \ge 0}$ as $g_\varepsilon(z^t)$ depends on $r^t$.  

Since $r^t$ is independent from previous picks $r^s$, $0 \le s \le t-1$, the conditional expectation $\bb E [\,z^{t+1}\,|\,\q F_t\,]$ can be evaluated as
\[
\bb E [\,z^{t+1}\,|\,\q F_t\,] 
= \bb E [\, f_{t+1}(r^0, \ldots, r^{t-1}, r^{t})\,|\,\q F_t\,]
= \sum_{r \in [R]^K} \left[\prod_{k=1}^{K} p_{r_k} \right] f_{t+1}(r^0, \ldots, r^{t-1}, r).
\]
Combining this property with \Cref{p: gradient properties} gives
\begin{equation}\label{eq: stoch grad cond exp}
\bb E [\,g_\varepsilon(z^t)\,|\,\q F_t \,] = \nabla_z \q L_{\varepsilon}(z^t)
\end{equation}
and
\begin{equation}\label{eq: stoch grad cond exp norm}
\bb E [\, \norm{g_\varepsilon(z^t)}_2^2 \,|\,\q F_t \,] \le \alpha [\q L_{\varepsilon}(z^t)  - \q L_{\varepsilon}^{\inf}] + \beta \norm{ \nabla_z \q L_{\varepsilon}(z^t) }_2^2  + \delta.
\end{equation}

With these definitions, we are prepared to state the descent lemma for \ref{eq: SAF}. 

\begin{lemma}[Descent lemma for \ref{eq: SAF}, based on \cite{Khaled.2020, Liu.2022}] \label{l: descent lemma SAF}
Let $\alpha, \beta$ and $\delta$ be defined as in \eqref{eq: abc params} and $\varepsilon \ge 0$. Consider a sequence $\{z_t \}_{t \ge 0}$ determined by \ref{eq: SAF} with an arbitrary starting point $z^0 \in \bb C^{d}$ and step sizes $\{ \mu_t \}_{t \ge 0}$ satisfying $\beta \norm{A}^2 \mu_t \le 1$. Then, we have
\begin{align*}
&\bb E [\q L_\varepsilon(z^{t+1}) - \q L_\varepsilon^{\inf} \,|\,\q F_t\,]  
\le (1 + \alpha \norm{A}^2 \mu_t^2) \left[ \q L_\varepsilon(z^t)  - \q L_\varepsilon^{\inf} \right] - \mu_t \norm{\nabla_z \q L_\varepsilon (z^t)}_2^2 + \delta \norm{A}^2 \mu_t^2.
\end{align*}
\end{lemma}    

\begin{proof}
An application of \Cref{l: descent lemma} with $v = - \mu_t g_\varepsilon(z^t)$ gives
\[
\q L_\varepsilon(z^{t+1}) = \q L_\varepsilon(z^{t} - \mu_t g_\varepsilon(z^t) ) \le \q L_\varepsilon(z^t) - 2 \mu_t \RE( \langle \nabla_z \q L_\varepsilon(z), g_\varepsilon(z^t) \rangle) + \norm{A}^2 \mu_t^2 \norm{g_\varepsilon(z^t)}_2^2.
\]
Averaging over the choice of the last step $r^t$, i.e., taking conditional expectation $\bb E [\, \cdot\,|\,\q F_t \,]$ yields
\begin{align*}
& \bb E [\q L_\varepsilon(z^{t+1}) \,|\,\q F_t\,] 
\le \q L_\varepsilon(z^t) - 2 \RE( \langle \nabla_z \q L_\varepsilon(z^t), \bb E[\, g_\varepsilon(z^t)\,|\,\q F_t\,] \rangle) + \norm{A}^2 \mu_t^2 \bb E\left[ \norm{g_\varepsilon(z^t)}_2^2\,|\,\q F_t\, \right] \\
& \quad \le \q L_\varepsilon(z^t) - 2 \mu_t \norm{ \nabla_z \q L_\varepsilon(z^t)}_2^2 + \norm{A}^2 \mu_t^2 \left( \alpha [\q L_{\varepsilon}(z^t)  - \q L_{\varepsilon}^{\inf}] + \beta \norm{ \nabla_z \q L_{\varepsilon}(z^t) }_2^2  + \delta \right),
\end{align*}
where we used \eqref{eq: stoch grad cond exp} and \eqref{eq: stoch grad cond exp norm}. By the assumption, we have $\beta \norm{A}^2 \mu_t \le 1$, which leads to
\begin{align*}
\bb E [\q L_\varepsilon(z^{t+1}) \,|\,\q F_t\,]  & = (1 + \alpha \norm{A}^2 \mu_t^2 ) [\q L_{\varepsilon}(z^t)  - \q L_{\varepsilon}^{\inf}] + \q L_{\varepsilon}^{\inf} - \mu_t \norm{ \nabla_z \q L_\varepsilon(z^t)}_2^2+ \delta \norm{A}^2 \mu_t^2.
\end{align*}
The result of \Cref{l: descent lemma SAF} is finalized by subtracting $\q L_{\varepsilon}^{\inf}$ from both sides. 
\end{proof}

With the descent lemma, we are equipped to prove the first convergence result for \ref{eq: SAF}.

\begin{proof}[Proof of \Cref{thm: SAF convergence constant}.]
The proof follows the steps of the proofs of \cite[Lemma 2, Theorem 2 and Corollary 1]{Khaled.2020}, but due to the differences between real and complex cases and alternations in one of the steps the resulting constants change. 
Since $\beta \norm{A}^2 \mu \le 1$, \Cref{l: descent lemma SAF} can be used.
Taking the expectation on both sides of \Cref{l: descent lemma SAF} leads to
\begin{equation}\label{eq: descent lemma SAF expectation}
\bb E [\q L_\varepsilon(z^{t+1}) - \q L_\varepsilon^{\inf}] \le (1 + \alpha \norm{A}^2 \mu_t^2) \bb E [\q L_\varepsilon(z^{t}) - \q L_\varepsilon^{\inf}] - \mu_t \bb E \norm{\nabla_z \q L_\varepsilon (z^t)}_2^2 + \delta \norm{A}^2 \mu_t^2.
\end{equation}

Let us recall that $\mu_t$ is constant, so that $\mu_t = \mu$ and consider notation $\Delta_t := \bb E[ \q L_\varepsilon(z^{t}) - \q L_\varepsilon^{\inf}]$ and $w = 1 + \alpha \norm{A}^2 \mu^2$. Then, multiplying \eqref{eq: descent lemma SAF expectation} with $w^{-t-1}$ gives 
\[
w^{-t-1} \Delta_{t+1} \le w^{-t-1} w \Delta_t - \mu w^{-t-1}  \bb E \norm{\nabla_z \q L_\varepsilon (z^t)}_2^2 + \delta \norm{A}^2 \mu^2 w^{-t-1}.
\]
By bringing the norm of the gradient to the left-hand side and summing up obtained inequalities for $t = 0,\ldots, T-1$ gives
\begin{align*}
& \mu \sum_{t = 0}^{T-1} w^{-t-1} \bb E \norm{\nabla_z \q L_\varepsilon (z^t)}_2^2 \le \sum_{t = 0}^{T-1} [w^{-t} \Delta_t - w^{-t-1} \Delta_{t+1} ] + \delta \norm{A}^2 \mu \sum_{t = 0}^{T-1} w^{-t-1} \\
& \quad = w^{0} \Delta_0 - w^{-T} \Delta_{T} + \delta \norm{A}^2 \mu^2 \sum_{t = 0}^{T-1} w^{-t-1} 
\le \Delta_0 + \delta \norm{A}^2 \mu^2 \sum_{t = 0}^{T-1} w^{-t-1},
\end{align*}
where we used that $\Delta_t := \bb E[ \q L_\varepsilon(z^{t}) - \q L_\varepsilon^{\inf}] \ge \q L_\varepsilon^{\inf} - \q L_\varepsilon^{\inf} = 0$.
Consequently, we obtain
\[
\min_{t = 0,\ldots, T-1} \bb E \norm{\nabla_z \q L_\varepsilon (z^t)}_2^2
\le \frac{\mu \sum_{t = 0}^{T-1} w^{-t-1} \bb E \norm{\nabla_z \q L_\varepsilon (z^t)}_2^2}{ \mu \sum_{t = 0}^{T-1} w^{-t-1} }
\le \frac{\Delta_0}{\mu \sum_{t = 0}^{T-1} w^{-t-1} } + \delta \norm{A}^2 \mu. 
\] 
If $\alpha =0$, which is only possible if $A$ is the zero matrix, $w = 1$ and $\sum_{t = 0}^{T-1} w^{-t-1} = T$. Otherwise, 
\[
\sum_{t = 0}^{T-1} w^{-t-1} 
= \frac{1 - w^{- T} }{w (1 - w^{-1}) }
= \frac{w^T - 1}{w^{T}(w - 1)}
\]
as a geometric sum. Thus, 
\[
\min_{t = 0,\ldots, T-1} \bb E \norm{\nabla_z \q L_\varepsilon (z^t)}_2^2
\le \frac{w^T (w-1) \Delta_0 }{\mu (w^T -1) } + \delta \norm{A}^2 \mu
= \frac{ \alpha \norm{A}^2 \mu^2 \Delta_0}{\mu}\left[ 1 + \frac{1}{w^T -1}\right] + \delta \norm{A}^2 \mu
\]  
Note that the function $s \mapsto 1 + \tfrac{1}{s}$ is decreasing on $(0,+\infty)$ and 
\[
w^{T} -1 = [1 + \alpha \norm{A}^2 \mu^2 ]^T -1 \ge T \alpha \norm{A}^2 \mu^2.
\]
Therefore, by combining these facts with the assumption $T \alpha \norm{A}^2 \mu^2 \le 1$ we arrive at 
\[
\min_{t = 0,\ldots, T-1} \bb E \norm{\nabla_z \q L_\varepsilon (z^t)}_2^2 
\le \frac{ \alpha \norm{A}^2 \mu^2 \Delta_0}{\mu}\left[ 1 + \frac{1}{T \alpha \norm{A}^2 \mu^2}\right] + \delta \norm{A}^2 \mu
\le \frac{ 2 \Delta_0}{T \mu} + \delta \norm{A}^2 \mu.
\]
Consequently, independently of $\alpha$, the bound above applies. The second term satisfies $\delta \norm{A}^2 \mu \le \gamma^2 / 2$ by the choice of $\mu$. Hence, selecting $T \ge \frac{ 4 \Delta_0}{\gamma^2 \mu}$ provides the desired bound on the norms of the gradient. If $\mu$ is chosen to be the minimum of three values, then 
\[
T \ge \frac{ 4 \Delta_0}{\gamma^2 \mu} = 
\frac{ 4 \Delta_0 }{\gamma^2  \min\left\{  \frac{1}{\sqrt{\alpha \norm{A}^2 T } } ,\frac{1}{\beta \norm{A}^2 }, \frac{\gamma^2}{ 2 \delta \norm{A}^2 } \right\} },
\] 
and by splitting the minimum into three separate cases and simplifying the first with $T$ on the right-hand side leads to the stated lower bound on $T$.    
\end{proof}

The next few proofs slightly expand on \cite{Liu.2022} and rely on the following result.

\begin{proposition}[{\cite[Theorem 1]{Robbins.1971}}]\label{p: near-supermartingale convergence}
Let $\{X_t\}_{t \ge 0}$, $\{Y_t\}_{t \ge 0}$ and $\{Z_t\}_{t \ge 0}$ be three sequences of nonnegative random variables that are adapted to a filtration $\{ \q F_t \}$. Let $\eta_t$ be a sequence of nonnegative real numbers such that $\sum_{t=0}^\infty \eta_t < \infty$. Suppose that 
\[
\bb E[Y_{t+1}\,|\, \q F_t \,] \le (1 + \eta_t) Y_t - X_t + Z_t, \quad \text{for all } t \ge 0,
\]
and $\sum_{t = 0}^\infty Z_t < \infty$ almost surely. Then, $\sum_{t = 0}^\infty X_t < \infty$ almost surely and $Y_t$ converges almost surely. 
\end{proposition}
 
A useful consequence of \Cref{p: near-supermartingale convergence} is the next lemma.
\begin{lemma}[{\cite[Lemma 3]{Liu.2022}}]\label{l: asymptotic rate}
Let $\{X_t\}_{t \ge 0}$ be a sequence of nonnegative real numbers and $\{ \mu_t\}_{t \ge 0}$ be a decreasing sequence such that $\sum_{t = 0}^\infty \mu_t X_t < \infty$ and \eqref{eq: step conditions 2} holds. Then, $\min_{t = 0, \ldots, T-1} X_t = o ([\sum_{t = 0}^{T-1} \mu_t]^{-1})$.
\end{lemma}

These two results combined provide \Cref{thm: SAF convergence decreasing}.
\begin{proof}[Proof of \Cref{thm: SAF convergence decreasing}.]
Let us apply \Cref{p: near-supermartingale convergence} with sequences 
\[
X_t = \mu_t \norm{\nabla_z \q L_\varepsilon(z^{t})}_2^2, 
\quad Y_t = \q L_\varepsilon(z^{t}) - \q L_\varepsilon^{\inf}, 
\quad Z_t = \delta \norm{A}^2 \mu_t^2,
\quad \eta_t = \alpha \norm{A}^2 \mu_t^2,
\] 
and filtration $\{\q F_t\}_{t \ge 0}$ as before. Then, $X_t, Y_t$ and $Z_t$ are adapted with $\q F_t$ and are nonnegative by definition. By the assumption \eqref{eq: step conditions}, we have
\[
\sum_{t=0}^\infty \eta_t = \alpha \norm{A}^2 \sum_{t=0}^\infty \mu_t^2 < \infty, 
\text{ and, analogously, } \sum_{t=0}^\infty Z_t < \infty.
\]
Finally, the last condition of \Cref{p: near-supermartingale convergence} is precisely \Cref{l: descent lemma SAF}, which holds under assumptions \eqref{eq: step conditions}. Hence, we get that $\q L_\varepsilon(z^{t}) - \q L_\varepsilon^{\inf}$ converges almost surely as $t \to \infty$ and $C^2 :=\sum_{t=0}^\infty \mu_t \norm{\nabla_z \q L_\varepsilon(z^{t})}_2^2 < \infty$ almost surely. Since $ \q L_\varepsilon^{\inf}$ is a constant, $\q L_\varepsilon(z^{t})$ converges almost surely. By definition of minimum, for any $T > s \ge 0$ we obtain
\begin{equation}\label{eq: lim inf tech}
\min_{t = s,\ldots, T-1} \norm{\nabla_z \q L_\varepsilon(z^{t})}_2^2 
\le \left[ \sum_{t=s}^{T-1}\mu_t \norm{\nabla_z \q L_\varepsilon(z^{t})}_2^2 \right] \cdot \left[ \sum_{t=s}^{T-1} \mu_t  \right]^{-1}
\le C^2 \left[ \sum_{t=s}^{T-1} \mu_t  \right]^{-1}.
\end{equation}
Since by \eqref{eq: step conditions}, the partial sums $\sum_{t=0}^{T-1} \mu_t \to \infty$ as $t \to \infty$, for $s=0$ it implies that 
\[
\min_{t = 0,\ldots, T-1} \norm{\nabla_z \q L_\varepsilon(z^{t})}_2^2 \to 0 \text{ as } t \to \infty.
\]
If additionally $\mu_t$ is decreasing and \eqref{eq: step conditions 2} holds, we apply \Cref{l: asymptotic rate}, which provides desired asymptotic convergence rate.
\end{proof}

The proof of \Cref{thm: SAF convergence decreasing expectation} is similar to the previous proof.

\begin{proof}[Proof of \Cref{thm: SAF convergence decreasing expectation}.]
The sequences $X_t, Y_t$ in the proof above are replaced by their expectations.
They are constant and adapted to any filtration $\{ \q F_t \}_{t \ge 0}$. Then, the proof is analogous with \Cref{l: descent lemma SAF} replaced by \eqref{eq: descent lemma SAF expectation}. Since everything is deterministic, almost sure convergence is just a regular convergence and $c$ is a constant satisfying $c^2 = \bb E C^2$. 
\end{proof}

Consequently, \Cref{cor: step size} follows by checking conditions of \Cref{thm: SAF convergence decreasing expectation}.
\begin{proof}[Proof of \Cref{cor: step size}.]
By the assumption on $\mu$ and the monotonicity of $\mu_t$, we have
\[
\beta \norm{A}^2 \mu_t \le \beta  \norm{A}^2 \mu \le 1. 
\] 
By the choice of $\theta$ and the integral test, the series $\sum_{t = 0}^{\infty} \mu_t$ is divergent and $\sum_{t = 0}^{\infty} \mu_t^2$ is convergent. Furthermore, $\mu_t$ is decreasing, which combined with $\theta < 1/2$ gives
\begin{align*}
\sum_{t = 0}^{T-1} \mu_t
& = \mu + \mu \sum_{t = 1}^{T-1} \int_{t-1}^t (1+t)^{-1/2-\theta} ds
\le \mu + \mu \sum_{t = 1}^{T-1} \int_{t-1}^t (1+s)^{-1/2-\theta} ds \\
& = \mu + \mu \int_0^{T-1} (1+s)^{-1/2-\theta} ds 
= \mu + \frac{\mu}{1 -1/2 -\theta} \left[ T^{1-1/2-\theta} - 1 \right] \\
& = \mu - \frac{\mu}{1/2 -\theta} + \frac{\mu}{1/2 -\theta} T^{1/2-\theta}
\le \frac{\mu}{1/2 -\theta} T^{1/2-\theta}.
\end{align*}  
Thus, we obtain  
\[
\frac{\mu_t}{\sum_{s = 0}^{t-1} \mu_s}
\ge \frac{1/2 -\theta}{(1 + t )^{1/2 + \theta} t ^{1/2 - \theta}}
\ge \frac{1/2 -\theta}{1 + t},
\] 
and by the comparison test the series $\sum_{t =0}^\infty \mu_t \left[ \sum_{s = 0}^{t-1} \mu_s \right]^{-1}$ is divergent. Hence, \Cref{thm: SAF convergence decreasing expectation} applies. Finally, we establish an upper bound on $[ \sum_{t=0}^{T-1} \mu_t ]^{-1}$ using the monotonicity of $\mu_t$,
\begin{align*}
\sum_{t = 0}^{T-1} \mu_t
& = \mu \sum_{t = 0}^{T-1} \int_t^{t+1} (1+t)^{-1/2-\theta} ds
\ge \mu \sum_{t = 0}^{T-1} \int_{t}^{t+1} (1+s)^{-1/2-\theta} ds \\
& = \mu \int_0^{T} (1+s)^{-1/2-\theta} ds 
= \frac{\mu}{1/2 -\theta} \left[ (1+T)^{1/2-\theta} - 1 \right].
\end{align*} 
Therefore, we have
\[
\min_{t = 0, \ldots, T-1} \bb E \norm{\nabla_z \q L_\varepsilon(z^{t})}_2^2 
\le c^2 \left[ \sum_{t=0}^{T-1} \mu_t  \right]^{-1} 
\le \frac{c^2(1/2 -\theta)}{\mu[(1+T)^{1/2-\theta} - 1]},
\]
and selecting $T$ as in the statement of corollary guarantees that $c^2(1/2 -\theta)\mu^{-1}[(1+T)^{1/2-\theta} - 1]^{-1} \le \gamma^2$.
\end{proof}

The next step is to translate the result of \Cref{thm: SAF convergence decreasing} in terms of $\liminf$. 
\begin{proof}[Proof of \Cref{col: min and liminf equivalence}.]
By definition of $\liminf$, we have
\[
\liminf_{t \to \infty} \norm{\nabla_z \q L_\varepsilon(z^{t})}_2
= \lim_{s \to \infty} \inf_{t \ge s} \norm{\nabla_z \q L_\varepsilon(z^{t})}_2.
\]
In the proof of \Cref{thm: SAF convergence decreasing}, we derived inequality \eqref{eq: lim inf tech}. Since by the assumption $\sum_{t=0}^\infty \mu_t$ diverges, which implies that $ \sum_{t=s}^\infty \mu_t$ diverges. Hence, taking limit $T \to \infty$ in \eqref{eq: lim inf tech} gives  $\inf_{t \ge s} \norm{\nabla_z \q L_\varepsilon(z^{t})}_2 = 0$ for all $s \ge 0$ and, consequently, $\liminf_{t \to \infty} \norm{\nabla_z \q L_\varepsilon(z^{t})}_2 = 0$ a. s.
\end{proof}

The last major proof in this section extends \cite[Theorem 2]{Orabona.2020} for a more general bound on expectation given by \Cref{p: gradient properties}. It is based on the next lemma.
\begin{lemma}[{\cite[Lemma 1]{Orabona.2020}}]\label{l: convergence to 0}
Let $\{\xi_t\}_{t \ge 0}$ and $\{\mu_t\}_{t \ge 0}$ be two nonnegative sequences and $\{ q^t \}_{t \ge 0}$ be a sequence in $\bb C^d$. Assume that 
\[
\sum_{t = 0}^{\infty} \mu_t \xi_t^p < \infty, \quad \sum_{t = 0}^{\infty} \mu_t = \infty \quad \text{ and } \quad \norm{ \sum_{t = 0}^{\infty} \mu_t q^t }_2 < \infty,
\]   
for some $p \ge 1$. If there exists $L \ge 0$ such that for all $t, k \ge 0$ the inequality
\[
| \xi_{t + k} - \xi_t | \le L \sum_{s = t}^{t + k -1}  \mu_s  \xi_s + L \norm{ \sum_{s = t}^{t + k -1}  \mu_s q^s }_2
\]
holds, then $\lim_{t \to \infty} \xi_{t} = 0$.
\end{lemma}

Now, we prove \Cref{thm: a.s. convergence of gradient}.

\begin{proof}[Proof of \Cref{thm: a.s. convergence of gradient}.]
Our goal is to apply \Cref{l: convergence to 0} with $\xi_t = \norm{\nabla_z \q L_\varepsilon(z^{t})}_2$ and $\mu_t$ being step sizes. Since $\varepsilon >0$, by \Cref{l: descent lemma} the gradient is Lipschitz continuous. Hence, for all $t, k \ge 0$ reverse triangle and triangle inequalities yield 
\begin{align*}
& \left| \norm{\nabla_z \q L_\varepsilon(z^{t+ k})}_2 - \norm{\nabla_z \q L_\varepsilon(z^{t})}_2 \right| 
\le \norm{\nabla_z \q L_\varepsilon(z^{t+ k})  - \nabla_z \q L_\varepsilon(z^{t}) }_2 
\le L \norm{ z^{t+ k} - z^{t} }_2   \\
& \quad = L \norm{ \sum_{s = t}^{t+k-1} \mu_s g_\varepsilon(z^s)}_2
\le L \norm{ \sum_{s = t}^{t+k-1} \mu_s \nabla_z \q L_\varepsilon(z^{s}) }_2 + L \norm{ \sum_{s = t}^{t+k-1} \mu_s [g_\varepsilon(z^s) - \nabla_z \q L_\varepsilon(z^{s})] }_2 \\
& \quad \le L \sum_{s = t}^{t+k-1} \mu_s \norm{\nabla_z \q L_\varepsilon(z^{s}) }_2 + L \norm{ \sum_{s = t}^{t+k-1} \mu_s [g_\varepsilon(z^s) - \nabla_z \q L_\varepsilon(z^{s})] }_2,
\end{align*}
where $L$ is the Lipschitz constant \eqref{eq: Lipschitz constant}.
Under assumptions of \Cref{thm: SAF convergence decreasing}, $\sum_{t = 0}^\infty \mu_t = \infty$. By \Cref{thm: SAF convergence decreasing expectation}, $\sum_{t =0}^\infty \mu_t \norm{\nabla_z \q L_\varepsilon(z^t)}_2^2 < \infty$ almost surely so that $p=2$ in \Cref{l: convergence to 0}. What is left is to show that $\norm{ \sum_{s=0}^\infty \mu_s [g_\varepsilon(z^s) - \nabla_z \q L_\varepsilon(z^{s})] }_2 < \infty$ almost surely.  Consider a random sequence 
\[
M_0 = 0 
\quad \text{and} \quad
M_t := \sum_{s=0}^{t-1} \mu_s [g_\varepsilon(z^s) - \nabla_z \q L_\varepsilon(z^{s})] \in \bb C^{d}, \quad t \ge 1.
\]
and filtration $\{ \q F_t \}_{t \ge 0}$ as before. By construction, $\{ M_t \}_{t \ge 0}$ is adapted to filtration $\{ \q F_t \}_{t \ge 0}$. Furthermore, by \eqref{eq: stoch grad cond exp} we have
\begin{align*}
\bb E[M_{t+1} \,|\, \q F_t ] 
& = \sum_{s=0}^{t-1} \mu_s [g_\varepsilon(z^s) - \nabla_z \q L_\varepsilon(z^{s})] +  \mu_{t} \bb E [g(z^t) - \nabla_z \q L_\varepsilon(z^{t}) \,|\, \q F_t ]  \\
& = \sum_{s=0}^{t-1} \mu_s [g_\varepsilon(z^s) - \nabla_z \q L_\varepsilon(z^{s})]
+ \mu_{t} [\nabla_z \q L_\varepsilon((z^t) - \nabla_z \q L_\varepsilon(z^{t}) ] \\
& = \sum_{s=0}^{t-1} \mu_s [g_\varepsilon(z^s) - \nabla_z \q L_\varepsilon(z^{s})] = M_{t},
\end{align*}
so that $M_t$ is a martingale and each of its components is a martingale. \Cref{p: gradient properties} also provides an upper bound on the expected squared norm of $g(z^t) - \nabla_z \q L_\varepsilon(z^{t})$. More precisely, using \eqref{eq: stoch grad cond exp}-\eqref{eq: stoch grad cond exp norm} and tower property we obtain
\begin{align*}
\bb E \norm{g(z^t) - \nabla_z \q L_\varepsilon(z^{t})}_2^2 
& = \bb E \left[ \bb E \left[ \norm{ g_\varepsilon(z^t) - \bb E\left[ g_\varepsilon(z^{t}) \,|\, \q F_{t} \right] }_2^2 \,|\, \q F_{t} \right] \right] \\
& = \bb E \left[ \bb E \left[ \norm{ g_\varepsilon(z^t) }_2^2 \,|\, \q F_{t} \right] - \norm{ \bb E \left[ g_\varepsilon(z^{t}) \,|\, \q F_{t} \right] }_2^2  \right] \\
& \le \bb E \left[ \alpha [\q L_\varepsilon(z^t) - \q L_\varepsilon^{\inf}] + \beta \norm{ \nabla_z \q L_\varepsilon(z^{t}) }_2^2 + \delta 
- \norm{ \nabla_z \q L_\varepsilon(z^{t})}_2^2 \right] \\
& = \alpha \bb E \left[\q L_\varepsilon(z^t) - \q L_\varepsilon^{\inf}\right] + (\beta-1) \bb E \norm{ \nabla_z \q L_\varepsilon(z^{t}) }_2^2 + \delta.
\end{align*}
Consequently, by \cite[Problem 13.27]{Athreya.2006}, the expected squared norm of $M_t$ is bounded by
\begin{align*}
\bb E \norm{M_t}_2^2  & = \sum_{s=0}^{t-1} \bb E \mu_s^2 \norm{g_\varepsilon(z^s) - \nabla_z \q L_\varepsilon(z^{s})}_2^2 \\
& \le  \sum_{s=0}^{t-1} \mu_s^2 \left( \alpha [ \bb E \q L_\varepsilon(z^s) - \q L_\varepsilon^{\inf}] + \delta \right) +  (\beta-1) \sum_{s=0}^{t-1} \mu_s^2 \bb E \norm{ \nabla_z \q L_\varepsilon(z^{s}) }_2^2. 
\end{align*} 
Since $\sum_{t =0}^\infty \mu_t^2$ is convergent, the sequence $\{ \mu_t \}_{t \ge 0}$ vanishes and is bounded. By \Cref{thm: SAF convergence decreasing expectation}, the sequence $\bb E \q L_\varepsilon(z^t)$ is bounded and $\sum_{s=0}^\infty \mu_s \bb E \norm{ \nabla_z \q L_\varepsilon(z^{s}) }_2^2 < \infty$. Denote by $\nu := \sup_{s \ge 0} \max\{ \mu_s, \bb E \q L_\varepsilon(z^s) \}$. Then,
\begin{align*}
& \bb E \norm{M_t}_2^2 \le \left[ \alpha [ \nu - \q L_\varepsilon^{\inf}] + \delta \right] \sum_{s=0}^{t-1} \mu_s^2 + \nu \max \{ \beta-1,0 \} \sum_{s=0}^{t-1} \mu_s \bb E \norm{ \nabla_z \q L_\varepsilon(z^{s}) }_2^2 \\
& \quad  \le \left[ \alpha [ \nu - \q L_\varepsilon^{\inf}] + \delta \right] \sum_{s=0}^\infty \mu_s^2 + \nu \max \{ \beta-1,0 \} \sum_{s=0}^\infty \mu_s \bb E \norm{ \nabla_z \q L_\varepsilon(z^{s}) }_2^2 < \infty.
\end{align*} 
For components of $(M_t)_j$, $j \in [d]$, this with Jensen's inequality \cite[Proposition 6.2.6]{Athreya.2006} yields 
\[
\sup_{t \ge 0} \bb E[ \max\{(M_t)_j , 0\}]
\le \sup_{t \ge 0} \bb E |(M_t)_j |
\le \sup_{t \ge 0} \bb E \norm{M_t}_2
\le \sup_{t \ge 0} \sqrt{\bb E \norm{M_t}_2^2} 
< \infty.
\]
Then, by \cite[Theorem 13.3.2]{Athreya.2006} each component $(M_t)_j$ converges to some random variable $M_j$ a.s. and $\bb E |M_j| < \infty$, which implies $M_j < \infty$ a.s. Let $M = \left( M_1, \ldots, M_d \right)^T$. Consequently, we get
\[
\norm{ \sum_{s=0}^\infty \mu_s [g_\varepsilon(z^s) - \nabla_z \q L_\varepsilon(z^{s})] }_2
= \norm{ \lim_{t \to \infty} M_t }_2 
= \norm{M}_2 
= \left[ \sum_{j =1}^d |M_j|^2 \right] < \infty \quad \text{a.s}.
\]
Finally, \Cref{l: convergence to 0} gives $\lim_{t \to \infty} \norm{\nabla_z \q L_\varepsilon(z^{t})}_2 = 0$ a.s.
\end{proof}

At last, we briefly discuss \Cref{rem: abs constants}.

\begin{proof}[Proof of \Cref{rem: abs constants}.]
For \Cref{rem: abs constants}, we can further bound $\q L_{\varepsilon,r}^{\inf} \ge 0$ in \eqref{eq: norm expectation bound}, which gives
\[
\bb E \norm{g_\varepsilon(z)}_2^2 \le \alpha \q L_{\varepsilon}(z) + \beta \norm{ \nabla_z \q L_{\varepsilon}(z) }_2^2.
\] 
Then, one can repeat the proofs in this section by considering sequence $\{ \q L_\varepsilon(z^{t}) \}_{t \ge 0}$ instead of $\{ \q L_\varepsilon(z^{t}) - \q L_\varepsilon^{\inf} \}_{t \ge 0}$ and treating $\delta$ as zero. This would lead to replacement of $\q L_\varepsilon(z^{0}) - \q L_\varepsilon^{\inf}$ with $\q L_\varepsilon(z^{0})$ in the resulting statements. 
\end{proof}

\subsection{Proof of \Cref{col: Kaczmarz} for Kaczmarz method}
\begin{proof}
The proof follows from \Cref{thm: SAF convergence constant} with
\[
\gamma = 4 \norm{A} \sqrt{ \q L(z^0)},
\] 
and $\mu_t = 1/ \norm{A}_F^2$. Let us start by restating constants $\alpha, \beta$ and $\delta$ in \eqref{eq: abc params}. In \eqref{eq: alpha kaczmarz}, we showed that  $\alpha = \norm{A}_F^2$. Since for the Kaczmarz method $K=1$, we have $\beta = 1 - 1/K = 0$. For $\delta$ we first observe that $\q L_{0,r} \ge 0$ and 
\[
\q L_{0,r}( \sqrt{y_r} \norm{A_{(r)}}_2^{-2} \conj{A_{(r)}} ) = \left| y^r \norm{A_{(r)}}_2^{-2} \left| A_{(r)}^T  \conj{A_{(r)}} \right| - \sqrt{y_r} \right|^2 =  |\sqrt{y_r} - \sqrt{y_r}\,|^2 = 0.
\] 
Thus, $\q L_{0,r}^{\inf} = 0$ and $\delta = \norm{A}_F^2 [\q L_0^{\inf} - \sum_{r \in [R]} \q L_{0,r}^{\inf} ]= 0$. Now, we need to check if all conditions of \Cref{thm: SAF convergence constant} are satisfied. The number of iterations $T$ satisfies 
\[
\frac{4 [ \q L_0(z^0)  - \q L_0^{\inf}]  }{\gamma^2 \mu} =
\frac{4 \norm{A}_F^2 \q L_0(z^0)}{16 \norm{A}^2 \q L_0(z^0) } \\
= \frac{\norm{A}_F^2}{4 \norm{A}^2 }
\le \left \lceil \frac{\norm{A}_F^2}{ 4\norm{A}^2 } \right \rceil
= T.
\]
Next, the inequality 
\[
\sqrt{ \alpha T} \norm{A}
= \norm{A}_F \norm{A} \left \lceil \frac{ \norm{A}_F^2} {4\norm{A}^2 } \right \rceil^{1/2}   
\le 2 \norm{A}_F \norm{A} \left[ \frac{ \norm{A}_F^2}{4 \norm{A}^2} \right]^{1/2}
= \norm{A}_F^2
\]
holds and the step size $\mu_t = 1/ \norm{A}_F^2$ admits $\mu_t \le (\sqrt{ \alpha T} \norm{A})^{-1}$. Hence, by \Cref{thm: SAF convergence constant}, 
\[
\min_{t = 0,\ldots, T-1} \bb E \norm{\nabla_z \q L_0(z^t)}_2 \le \gamma =  4 \norm{A} \sqrt{ \q L_0(z^0)}.
\]
\end{proof}

\subsection{Proofs for Ptychographic Iterative Engine }

We start with the proof of \Cref{thm: PIE as stochastic gradient}.

\begin{proof}[Proof of \Cref{thm: PIE as stochastic gradient}.]
Recall that by \eqref{eq: STFT matrix} the matrix $A_{r^t}$ admits
\[
A_{r^t} z = F [S_{s_t} w \circ z ] = F \diag(S_{s_t} w) z. 
\]
With a help of \eqref{eq: fourier prop}, we rewrite the iteration of \ref{eq: PIE} as 
\begin{align*}
z^{t+1} & = z^{t} + \frac{\alpha_t \diag(\conj{S_{s_t}  w})}{\norm{w}_\infty^2} \left[ \frac{1}{d} F^* \diag\left( \frac{ \sqrt{ y^{r^t} }}{ |F [S_{s_{t}} w \circ z^t]|} \right) F -  \frac{1}{d} F^* F  \right] \diag(S_{s_t} w) z^t  \nonumber \\
& = z^{t} + \frac{\alpha_t}{d \norm{w}_\infty^2} \diag(\conj{ S_{s_t} w}) F^* \left[ \diag\left( \frac{ \sqrt{ y^{r^t} } }{ |F \diag(S_{s_{t}} w) z^t]|} \right)  -  I_d \right] F \diag(S_{s_t} w) z^t  \nonumber \\
& = z^{t} + \frac{\alpha_t}{d \norm{w}_\infty^2}  A_{r^t}^* \left[ \diag\left( \frac{ \sqrt{ y^{r^t} }}{ | A_{r^t} z^t | } \right) - I_d \right] A_{r^t} z^t \\
& = z^t -  \frac{\alpha_t}{d \norm{w}_\infty^2} A_{r^t}^* \left[ A_{r^t} z^t - \sqrt{ y^{r^t} } \sgn( A_{r^t} z^t ) \right] \\
& = z^t - \frac{\alpha_t}{d \norm{w}_\infty^2} \nabla_z \q L_{0, r^t}(z^t)
= z^t - \frac{\alpha_t p_{r^t} }{d \norm{w}_\infty^2} g_0(z^t),
\end{align*}
where in the last line we used \eqref{eq: stochastic gradient} with $K = 1$. 

The step size of \ref{eq: PIE} is, in fact, analogous to the step size of the Kaczmarz method. To see this, let us compute $\norm{A_r}^2$. Using the definition of circular shift matrix \eqref{eq: shift op}, \eqref{eq: fourier prop} and properties of diagonal matrices, we obtain
\begin{align*}
\norm{A_r}^2
& = \norm{A_r^* A_r} 
= \norm{\diag(\conj{ S_{s_t} w} ) F^*  F \diag(S_{s_t} w) } \\
& = d \norm{\diag( |S_{s_t} w|^2 ) }
= d \norm{S_{s_t} w}_\infty^2
= d \norm{w}_\infty^2.
\end{align*}
Consequently, the step size $\mu_t$ is given by $\alpha_t p_{r^t}/\norm{A_{r^t}}^2$.
\end{proof}

Next, we combine \Cref{thm: PIE as stochastic gradient} with \Cref{cor: step size}.

\begin{proof}[Proof of \Cref{col: PIE convergence}.]
Substituting uniform probabilities $p_r = 1/ R$ and the choice of $\alpha_t = \alpha / (1 + t)^{1/2 + \theta}$ with $0 < \theta < 1/2$ leads to 
\[
\mu_t = \frac{\alpha_t p_{r^t} }{d \norm{w}_\infty^2} = \frac{\alpha }{ d R \norm{w}^2 (1 + t)^{1/2 + \theta}} =: \frac{\mu(\alpha)}{(1 + t)^{1/2 + \theta}}.
\]
Since $K=1$, i.e., only one index is sampled for the stochastic gradient, $\beta$ in \eqref{eq: abc params} is equal to zero. Thus, condition $\beta \norm{A}^2 \mu(\alpha) < 1$ is satisfied for all $\alpha >0$. Hence, we apply an analogue of \Cref{cor: step size} for almost sure convergence, which gives us desired convergence rate of $\min_{t = 0,\ldots, T-1} \norm{\nabla \q L(z^t)}_2$. 
\end{proof}

\section{Conclusions}

In this paper, we considered the Stochastic Amplitude Flow (SAF) algorithm for phase retrieval and derived its convergence guarantees. These results are also applicable to the Kaczmarz method and Ptychographic Iterative Engine as special cases of SAF. In particular, for the latter algorithm, we obtained sufficient conditions on parameters $\alpha_t$ so that the algorithm converges, which is a known issue among practitioners. As a next step, we aim to expand this analysis for extended Ptychographic Iterative Engine, a version of the algorithm for the case when not only an object $x$ has to be recovered, but also a vector $w$ in \eqref{eq: ptycho meas}, which is commonly unknown in practical applications.

\section*{Acknowledgment}
The author thanks Benedikt Diederichs, Frank Filbir and Patricia R\"{o}mer for helpful discussions. This work was funded by the Helmholtz Association under contracts No.~ZT-I-0025 (Ptychography 4.0), No.~ZT-I-PF-4-018 (AsoftXm), No.~ZT-I-PF-5-28 (EDARTI), No.~ZT-I-PF-4-024 (BRLEMMM). 

\bibliography{stochastic_AF_and_pie_convergence.bbl}

\end{document}